\documentclass[11pt,a4paper]{amsart} 
\usepackage[centering,left=3.0cm,right=3.0cm]{geometry}
\usepackage{url}
\usepackage{mathrsfs}
\usepackage{enumitem}
\usepackage{color} 
\usepackage[colorlinks,linkcolor=blue,anchorcolor=blue,citecolor=blue,backref=page]{hyperref}
\usepackage{graphics,epsfig}
\usepackage{graphicx}
\usepackage{float}
\usepackage{epstopdf}
\usepackage[utf8]{inputenc}
\usepackage{hyperref}
\usepackage{calc, amscd}
\usepackage[T1]{fontenc}
\usepackage{pdfpages}
\usepackage{mathtools} 
\usepackage{scalerel,stackengine}
\hypersetup{breaklinks=true}
\usepackage{amsthm,mathrsfs,multirow,xcolor,lscape,longtable,pbox,lipsum}
\usepackage{amsfonts,amscd,amssymb}
\usepackage[thinlines]{easytable}
\usepackage{microtype}
\usepackage{todonotes}
\usepackage{phaistos}
\newtheorem{theorem}{Theorem}
\newtheorem{lemma}[theorem]{Lemma}
\newtheorem{corollary}[theorem]{Corollary}

\usepackage{mathtools}

\DeclarePairedDelimiter\floor{\lfloor}{\rfloor}

\newcommand{\cP}{{\mathcal P}}

\newcommand{\rF}{{\mathbb F}}

\newcommand{\rZ}{{\mathbb Z}}
\newcommand{\rP}{{\mathbb P}}
\newcommand{\ep}[1]{{\mathbf e}_p\left({#1}\right)}
\newcommand{\el}[1]{{\mathbf e}_{\ell}\left({#1}\right)}

\newcommand{\ud}{{\,\mathrm d}}

\newcommand{\ord}[1]{\,\mathrm{ord} \, {#1}}
\newcommand{\tr}[1]{\,\mathrm{Tr}\left({#1}\right)}
\newcommand{\Tr}[1]{\,\mathrm{Tr}_{#1}}
\newcommand{\eqr}[1]{\mbox{(\ref{eq:#1})}}

\newcommand{\ie}{i.e.\ }
\def\exp{\mathrm{exp}}

\newcommand{\tmt}[4]{\left({#1\atop #3}{#2\atop #4}\right)}

\newcommand{\mr}[1]{\mathrm{#1}}
\newcommand{\inner}[1]{\left\langle{#1}\right\rangle}
\newcommand\rwidehat[1]{\ThisStyle{%
  \setbox0=\hbox{$\SavedStyle#1$}%
  \stackengine{-1.0\ht0+.5pt}{$\SavedStyle#1$}{%
    \stretchto{\scaleto{\SavedStyle\mkern.15mu\char'136}{2.6\wd0}}{1.4\ht0}%
  }{O}{c}{F}{T}{S}%
}}

\begin{document}
\date{\today}
\title[exponential sums associated to modular forms]{exponential sums in prime fields for modular forms}
\author{Jitendra Bajpai, Subham Bhakta and Victor~C.~Garc\'{i}a }
\address{J.~Bajpai : Mathematisches Institut, Georg-August-Universit\"at G\"ottingen, Germany.} 
\email{jitendra@math.uni-goettingen.de}
\curraddr{Institut f\"ur Geometrie, Technische Universit\"at Dresden, Germany}
\email{jitendra.bajpai@tu-dresden.de}
\address{S. Bhakta: Mathematisches Institut, Georg-August-Universit\"at G\"ottingen, Germany.} \email{subham.bhakta@mathematik.uni-goettingen.de }
\address{V.~C.~Garc\'{i}a : Universidad Aut\'onoma Metropolitana, M\'exico.} \email{vcgh@azc.uam.mx}

\begin{abstract}
The main objective  of this article is to study the exponential sums associated to Fourier coefficients of modular forms supported at 
numbers having a fixed set of prime factors. This is achieved by establishing an improvement on Shparlinski's bound for exponential sums attached 
to certain linear recurrence sequences over finite fields. 
\end{abstract}

\maketitle

\tableofcontents

\section{Introduction} 
Let $f$ be a modular form of weight $k \in 2\rZ$
    and level $N$ such that it has a Fourier expansion 
  \begin{equation*}\label{eq:mf} 
    f(z)=\sum_{n=1}^{\infty}a(n) e^{2\pi i n z}, \quad \Im (z) \ge 0,
  \end{equation*}
    with $a(n)$ be the $n^{th}$ Fourier coefficient. In this article, we shall restrict to the family of  modular forms with rational coefficients, 
    that is,  $f(z)$ with $a(n)\in \mathbb{Q}$ 
    for every $n$. We first consider Hecke eigenforms or simply eigenforms in the space of  cusp forms of weight $k$ for the congruence subgroup 
    $\Gamma_1(N)$ with trivial nebentypus. When $f$ is an eigenform with integer Fourier coefficients, it follows from Deligne-Serre 
    that for any prime $\ell,$ there exist a corresponding Galois representation
  \[ 
    \rho_f^{(\ell)}:\mathrm{Gal}\left(\overline{\mathbb{Q}}/\mathbb{Q}\right) \longrightarrow
    \mathrm{GL}_2\left(\mathbb{Z}_{\ell}\right)
  \]
    such that $\text{tr}(\rho_f^{(\ell)}(\text{Frob}_p))=a(p),$ for any prime $p \nmid N\ell.$ For quick reference on the Deligne-Serre correspondence, we refer the interested reader to~\cite[Chapter 3]{Ferraguti}. In particular, 
    $a(p)\hspace{-0.1cm}\pmod\ell$ is determined by the trace of the corresponding Frobenius element in
  \[ 
    \text{GL}_2(\mathbb{Z}_{\ell}/\ell\mathbb{Z}_{\ell})=\text{GL}_2(\mathbb{Z}/\ell\mathbb{Z}).
  \] 
    In certain cases, Chebotarev's density 
    theorem implies that given any $\lambda \in \mathbb{F}_{\ell},$ there exists a prime $p$ such that $a(p)\equiv\lambda\hspace{-0.1cm} \pmod{\ell}.$ However, the set of such primes comes with density strictly less than $1$. Now one may naturally ask, whether there is an absolute constant $s$ such that for any given primes $p$ and $\ell,$ and any element $\lambda \in \mathbb{F}_{\ell},$ the equation 
  \[
    \sum_{i=1}^{s} a\left(p^{n_{(i,p)}}\right)\equiv \lambda\hspace{-0.3cm}\pmod\ell
  \]
    is solvable for some tuple $\left(n_{(i,p)}\right)_{1 \leq i\leq  s}$ of positive integers? In other words, does there exist an absolute constant $s$ such that, given primes $p$ and $\ell,$ any element of $\mathbb{F}_{\ell}$ can be written as sum of at most $s$ elements of the set $\{a(p^n)\}_{n \geq 0}?$ 
    
    Let $\tau(n)$ be the Ramanujan function, 
    which is defined by the identity
  \[
    \Delta(z)=q\prod_{n\geq 1}(1-q^n)^{24} = \sum_{n\geq 1}\tau(n)q^n, \quad \textrm{with}\,\,  q=\exp(2 \pi i z).
  \]  
    In~\cite{Shparlinski2005}, Shparlinski proved that the set $\{\tau(n)\}$ is an additive basis modulo any prime $\ell,$ that is
    there is an absolute constant $s$ such that the Waring-type congruence
  \[
    \tau(n_1)+\cdots+ \tau(n_s) \equiv \lambda \hspace{-0.3cm}\pmod{\ell}
  \]    
    is solvable for any residue class $\lambda\hspace{-0.1cm} \pmod{\ell}.$ In~\cite{GGK08}, Garaev, Garc\'ia and Konyagin  proved that for 
    any $\lambda \in \mathbb{Z}$, the equation
  \[
    \sum_{i=1}^{s}\tau(n_i)=\lambda
  \]
    always has a solution for $s=74,000.$ 

    Later Garc\'ia and Nicolae~\cite{GN18} extended this result for coefficients $a(n)$
    of normalized Hecke eigenforms of weight $k$ in $S_k^{\textrm{new}}(\Gamma_0(N))$. 
    More precisely, they proved that for any $\lambda \in \mathbb{Z}$,
    the equation
  \[ 
    \sum_{i=1}^{s}a(n_i)=\lambda
  \]
    always has a solution for  some $s \leq c(f)$ and $c(f)$ satisfying
  \[
    c(f)\ll (2N^{3/8})^{\tfrac{k-1}{2} + \varepsilon} k^{\tfrac{3}{16}k + O(1) + \varepsilon} \log(k+1).
  \]    

    The proof of the above two results are connected to the identity $a(p^2)= a^2(p)-p^{k-1}$ with
    the solubility of the  equation
  \[
    p^{k-1}_1 + \cdots + p^{k-1}_s = N, \quad \textrm{for primes }\, p_1,\ldots , p_s.
  \]

    We are actually studying a closely related problem over finite 
    fields, and our main tool is Theorem~\ref{Thm:Main} which provides 
    a nontrivial bound for exponential sums with coefficients of modular forms. In other words, in Theorem~\ref{thm:main1-1} and Theorem~\ref{thm:main1-2} we are generalizing Shparlinski's result to a wider class of modular forms. We shall mainly see in which cases the $n_i$'s can be taken to be the powers of a given prime, and we record this in Corollary~\ref{le:later} and Corollary~\ref{le:later2}. To study this problem, we shall primarily focus on the exponential sums of type
     \begin{equation*}
    \max_{\xi \in \rF_{\ell}^*} \left|\sum_{n\le \tau}\el{\xi a(p^n) }\right|
  \end{equation*}
where $p,\ell$ are primes, and $\tau$ is a suitable parameter which we shall specify later. Moreover, we shall also study such exponential sums for certain cusp forms which are not necessarily eigenforms.

 When $f$ is a normalized eigenform, it is well known that $a(n)$ is a multiplicative function and for any prime $p\nmid N$ satisfies the relation
  \begin{equation}\label{identity:primepowers}
    a(p^{n+2})=a(p)a(p^{n+1})-p^{k-1}a(p^n),\quad  n\ge 0.
  \end{equation} 
    Moreover, we have $a(p^n)=a(p)^n$ for any prime $p\mid N$. These facts come from the properties of Hecke operators, see ~\cite[Proposition 5.8.5]{DS2005}. If $a(p) \in \mathbb{Q},$ then we can consider $a(p) \hspace{-0.1cm}\pmod \ell \in \mathbb{F}_{\ell}$ naturally. We shall shortly give a brief review of linear recurrence sequences. On the other hand, any cuspform can be uniquely written as a $\mathbb{C}$-linear combination of pairwise orthogonal eigenforms 
    with Fourier coefficients coming from $\mathbb{C}$. See~\cite[Chapter 5]{DS2005} for a brief 
    review of the Hecke theory of modular forms. Here we are concerned  with all such eigenforms having rational coefficients. In this case, the sequence $\{a(p^n)\}$ is still a linear recurrence sequence of possibly higher degree.

 \subsection{Linear recurrence sequences and Shparlinski's bound}\label{se:lrs}
We now provide a quick overview of the basic theory of linear recurrence sequences.
Let $r\ge 1$ be an integer and $p$ be an arbitrary  prime number.
    A \emph{linear recurrence sequence} $\{s_n\}$ of order $r$ in $\rF_p$ consists of 
    a recursive relation
  \begin{equation}\label{eq:RecurrenceSeq}
    s_{n+r}\equiv a_{r-1} s_{n+r-1} + \cdots + a_0 s_n \hspace{-0.2cm}\pmod p, \quad \textrm{with } n=0,1,2, \ldots \,,
  \end{equation}
    and initial values $s_0, \ldots, s_{r-1} \in \rF_p.$ 
    Here $a_0, \ldots, a_{r-1} \in \rF_p$ are fixed. The 
    {\it{characteristic polynomial $\omega(x)$ associated to}} $\{s_n\}$ is   
  \[
    \omega(x)=x^{r}- a_{r-1}x^{r-1}- \cdots - a_1 x - a_{0}.
  \]    

    Under certain assumptions, linear recurrence sequences become periodic modulo $p$,  
    see~\cite[Lemma 6.4]{Korobov1992} and~\cite[Theorem 6.11]{Lidl-Niederreiter}.

    Let $p$ be a prime number and $\omega(x)$ be the characteristic polynomial of a linear 
    recurrence sequence $\{ s_n\}$ defined by equation~\eqr{RecurrenceSeq}. If 
    $(a_0,p)=1$ and at least one of the  $s_0, \ldots, s_{r-1}$ are not divisible by 
    $p,$ then the sequence $\{s_n\}$ is periodic modulo $p,$ that is for some $T\ge 1$, 
   \[
      s_{n+T}\equiv s_n \hspace{-0.2cm}\pmod p,  \qquad n=0,1,2,\ldots\, .
   \] 
    The least  positive period is denoted by $\tau.$ Moreover, $ \tau \le p^r-1$ and 
    $\tau$ divides $T$ for any period $T \ge 1$ of the sequence $\{s_n\}.$

    In 1953, Korobov~\cite{Korobov1953} obtained bounds for rational exponential sums
    involving linear recurrence sequences in residue classes. In particular, 
    for the fields of order $p,$ if 
    $\{s_n\}$ is a linear recurrence sequence of order $r$ with $(a_0,p)=1$
    and period $\tau$, it follows that
  \begin{equation}\label{eq:Korobov}
    \left|\sum_{n\le \tau} \ep{s_n}\right|\le p^{r/2}.
  \end{equation}  
    
    Note that such a bound is nontrivial if 
    $p^{r/2}< \tau$ and  asymptotically effective only if $p^{r/2} / \tau \to 0$ as $p\to \infty.$ Estimate~\eqref{eq:Korobov} is optimal in general terms, indeed Korobov~\cite{Korobov1992} showed that there is a linear recurrence sequence $\{s_n\}$
    with length $r$ satisfying
  \[
        \frac{1}{2} p^{r/2} < \left|\sum_{n\le \tau} \ep{s_n}\right|\le p^{r/2}.
  \]
    In turn, it has been proved that there exists a class of linear recurrence sequences with a better upper bound
  \[
    \left|\sum_{n\le \tau} \ep{s_n}\right|\le \tau^{1/2+ \varepsilon}.
  \]
    However, the proof of the existence is ineffective in the sense that we do not know any
    explicit characteristics of such family, see~\cite[Section 5.1]{EGPASWT2003}.\\

    The case when the associated polynomial $\omega(x)$ is irreducible in $\rF_p[x],$ was widely studied.
    For instance, from a more general result due to Katz \cite[Theorem 4.1.1.]{Katz1988} 
    it follows that if $\omega(0)=1$ then
  \[
    \left|\sum_{n \le \tau} \ep{s_n} \right|\le p^{(r-1)/2}.
  \]    
    Shparlinski~\cite{Shparlinski2004} improved Korobov's bound for all nonzero linear recurrence sequences
    with irreducible characteristic polynomial $\omega(x) $ in $\rF_p[x].$ From~\cite[Theorem 3.1]{Shparlinski2004} we get
  \[
    \max_{\xi \in \rF_p^*} \left|\sum_{n \le \tau}\ep{\xi s_n} \right| \le  
      \tau { p^{-\varepsilon/(r-1)}}+ r^{3/11}\tau^{8/11}p^{(3r-1)/22},
  \]    
    with period $\tau$ satisfying that
  \begin{equation}\label{eq:conditions}
     \max_{\substack{ d< r \\ d| r}}  \gcd(\tau, p^d-1) < \tau p^{-\varepsilon}. 
  \end{equation}
    In particular, if $r$ is fixed then the upper bound is non trivial 
    for $\tau \ge p^{r/2 - 1/6 + \varepsilon}.$ We would like to point out that the
    condition \eqref{eq:conditions} above is essential if $\tau \le p^{r/2 + \varepsilon}$, for details see the example in~\cite[Section 1]{Shparlinski2004}. Moreover, we consider the general case when the associated polynomial $\omega(x)$
    is not necessarily irreducible, and we deduce the following key result.
   
    \begin{theorem}\label{Thm:Main}
      Let $p$ be a large prime number and $\varepsilon  > \varepsilon' > 0.$  Suppose that  $\{ s_n\}$ is 
      a nonzero linear recurrence sequence 
     with positive order  and period $\tau$ in $\rF_p$ such that its characteristic polynomial 
     $\omega(x)$ has distinct roots in its
     splitting field, and $(\omega(0), p)=1$. Set $\omega(x)=\prod_{i}^{\nu}\omega_i(x)$ as a product of
     distinct irreducible polynomials
     in $\rF_p[x],$ and for each $i,$ $\alpha_i$ denotes a root of $\omega_{i}(x).$ 
     If all polynomials $\omega_i(x)$ have the same degree, \ie $\deg \omega_i(x)=r>1,$ 
     and the system $\tau_i=\ord{\alpha_i},$  
     satisfies
   \begin{align}\label{eq:RootsOrder}
& \textbf{a) } \, \max_{\substack{ d< r \\ d| r}} \gcd (\tau_i, p^d-1) 
                < \tau_i p^{-\varepsilon}, \quad \textrm{at least for one } 1 \le i \le \nu, \\
&\textbf{b) } \, \gcd (\tau_i,\tau_j) < p^{\varepsilon'},~\textrm{for~some~pair}~i \neq j~{along~with}~\mathbb{F}_p(\alpha_i)\cong\mathbb{F}_p(\alpha_j),\nonumber
   \end{align} then there exists a $\delta=\delta(\varepsilon,\varepsilon')>0$ such that
   \begin{equation}\label{eq:t1}
     \max_{\xi \in \rF_p^*} \left|\sum_{n \le \tau}\ep{\xi s_n} \right| \le   \tau { p^{-\delta}}.
    \end{equation}
    \end{theorem}

    This generalizes~\cite[Corollary]{Bourgain2005a} due to Bourgain, where all of the irreducible factors have degree $r=1$ while Theorem~\ref{Thm:Main} deals with the case $r\ge 2.$
    This will be of immense use in what follows, roughly because the characteristic polynomial associated to $\{a(p^n)\}$ have degree two.
    
    Theorem~\ref{Thm:Main} will be essential to establish Theorem~\ref{thm:main1-1} and Corollaries \ref{coro:Waring} 
    and ~\ref{coro:Non-linearity}.  
    Our approach, which relies on the sum-product phenomenon, provides an improvement over Shparlinski's Theorem 3.1 of \cite{Shparlinski2004} 
    for the same class of linear recurrence sequences, obtaining non trivial exponential sums in a larger range. To be more precise, if $p(r)$ denotes the least prime divisor of $r$ then any $\tau>p^{r/p(r)+\varepsilon}$ satisfies 
  \[
    \tau p^{-\varepsilon} > p^{r/p(r)} \geq \max_{\substack{d<r\\ d|r}}\gcd(\tau, p^{d}-1).
  \]
    In particular, our result works for any $\tau>p^{r/p(r)+\varepsilon},$ while Shparlinski's bound in \cite{Shparlinski2004} is nontrivial if $\tau>p^{r/2-1/6+\varepsilon}.$ This is an improvement if $p(r)>2,$ more precisely when $r$ is odd.

\subsection{Main results}
We now quickly discuss the main results obtained in this article. In the list, our first result is  the following:
     
\begin{theorem}\label{thm:main1-1}
Let $f(z)$ be an eigenform with rational coefficients $a(n)$. Let $\cP$ be the set of primes $p$ such that $a(p^u)\neq 0$ for any $u \in \mathbb{N}.$
Then the following is true.
\begin{enumerate}  
\item[(i)] The set of primes $\cP$ satisfies that given $p \in \cP,$
    for any $0 < \varepsilon < 1/2$ there exists a $\delta=\delta(\varepsilon)>0$ such that the following estimate
     \begin{equation}\label{eq:main1}
    \max_{\xi \in \rF_{\ell}^*} \left|\sum_{n\le \tau}\el{\xi a(p^n) }\right| \leq 
    \tau \ell^{-\delta},
  \end{equation}
    holds for $\pi(y)+O(y^{2\varepsilon})$ many primes $\ell \leq y,$ where the least period $\tau$  of the linear recurrence sequence $\{a(p^n)\} \hspace{-0.1cm}\pmod\ell$, depends on both $p$ and $\ell$. Here $\pi(y)$ denotes the number of primes up to $y$, which is asymptotically equivalent to $\frac{y}{\log y}.$
\item[(ii)]  For the exceptional set of primes $p \notin \cP,$ let $u$ be the least natural number such that $a(p^u)=0.$ Then for any $ 0 <\varepsilon < 1/2,$ there exists a $\delta=\delta(\varepsilon)>0$ such that the following estimate
  \begin{equation}\label{eq:main111} 
    \max_{\xi \in \rF_{\ell}^*} \left|\sum_{n\le \tau}\el{\xi a(p^n) }\right| =
    \frac{\tau}{u+1}+O(\tau\ell^{-\delta}+u).
    \end{equation}
    holds for $\pi(y)+O(y^{2\varepsilon})$ many primes $\ell \leq y.$
    \end{enumerate}
    \end{theorem}
    Roughly speaking, a newform of level $N$ is a normalized eigenform which is not a cuspform of level $N'$ for any proper divisor $N'$ of $N.$ For details and basics on modular forms, we refer the reader to~\cite{DS2005}. A newform is said to have complex multiplication ($CM$) by a quadratic Dirichlet character 
    $\phi$ if $f = f \otimes \phi$, where we define the twist as
    \[f \otimes \phi = \sum_{n=1}^{\infty} a(n)\phi(n)q^n.\] 
    In part $(i)$ of Theorem~\ref{thm:main1-1}, the condition $a(p^u)\neq 0$ holds for almost all prime $p$ provided that $f$ is a newform without $CM.$ This is a consequence of Sato-Tate conjecture and we shall discuss this again in the proof of Lemma~\ref{int:density1}. In particular, we have a non trivial estimate for the following exponential sum \begin{equation}\label{eq:main112} 
    \max_{\xi \in \rF_{\ell}^*} \left|\sum_{n\le \tau}\el{\xi a(p^n) }\right| \,.
    \end{equation}
    Let us recall that any general cusp form $f$ can be uniquely written as $\mathbb{C}$-linear combination of eigenforms. We call these eigenforms as components of $f$. We then have the following result.
    
     \begin{theorem}\label{thm:main1-2} Let $f(z)$ be a cusp form which is not necessarily an eigenform, and can be written as a $\mathbb{Q}$-linear combination of newforms with rational coefficients. Suppose that there are $r_2$ many components with CM, then under the assumption of GST hypothesis\footnote{See Section~\ref{se:gst} for the discussion about GST hypothesis.} there exist a set of primes $p$ with density at least $2^{-r_2}$ such that for any $0 < \varepsilon < 1/2$ there exists a $\delta=\delta(\varepsilon)>0$ for which the following estimate
     \begin{equation}\label{eq:main5}
    \max_{\xi \in \rF_{\ell}^*} \left|\sum_{n\le \tau}\el{\xi a(p^n) }\right| \leq 
    \tau \ell^{-\delta},
  \end{equation}
    holds for $c_f\pi(y)+O(y^{2\varepsilon})$ many primes $\ell \leq y,$ where $c_f>0$ is a constant. 
    \end{theorem}
 In both of the theorems above, we took a fixed prime $p$ and looked for primes $\ell$ for which a non trivial estimate to~(\ref{eq:main112}) holds. However, these results are valid for almost all prime $\ell,$ and we do not know explicitly which of the primes are being excluded in this process. So one may now naturally ask, what if we now fix a prime $\ell$ and find out for how many primes $p$ the sum at (\ref{eq:main112}) is non trivial. In this regard, we have the following results.
   
\begin{theorem}\label{thm:main2-1} 
 Let $f(z)$ be a newform of weight $k,$ without CM, and with integer Fourier coefficients. Consider the set $\mathfrak{P}=\left\{\ell~\mathrm{prime} \mid (k-1,\ell-1)=1\right\}.$ Then for any fixed $\varepsilon>0$ and any large enough $\ell \in \mathfrak{P},$ the set of primes $p$ satisfying
  \begin{equation}\label{eq:hh}
    \max_{\xi \in \rF_{\ell}^*} \left|\sum_{n\le \tau}\el{\xi a(p^n) }\right|\le \tau \ell^{-\delta} 
  \end{equation}
    have density at least $1+O_{\varepsilon}\left(\frac{1}{\ell^{1-3\varepsilon}}\right),$ where $\delta=\delta(\varepsilon)$ is same as in Theorem~\ref{Thm:Main}. 
    \end{theorem}
Intuitively, we can regard this theorem as the inverse of Theorem~\ref{thm:main1-1}, and in this analogy, we have the following result which can be regarded as the inverse of Theorem~\ref{thm:main1-2}. Just for the sake of simplicity we are assuming $(k-1,\ell-1)=1,$ which can be easily avoided and we will make it evident from the proof of the following theorem.
    \begin{theorem}\label{thm:main2-2} 
If $f(z)$ is a cuspform, and can be written as $\mathbb{Q}$ linear combination of $r$ many newforms without CM and with integer coefficients, such that all of these components satisfies GST hypothesis. Then for any fixed $\varepsilon>0$ and large enough $\ell,$ the set of primes $p$ satisfying
  \begin{equation}\label{eq:hh1}
    \max_{\xi \in \rF_{\ell}^*} \left|\sum_{n\le \tau}\el{\xi a(p^n) }\right|\le \tau \ell^{-\delta} 
  \end{equation}
    have density at least $2^{-r}+O_{\varepsilon}\left(\frac{1}{\ell^{1-2\varepsilon}}\right),$ where $\delta=\delta(\varepsilon)$ is same as in Theorem~\ref{Thm:Main}. 
    \end{theorem}

\section{Exponential sums with linear recurrence sequences}\label{KB}
    In this section, our main goal is to prove Theorem~\ref{Thm:Main}, which is one of our key tool in 
    establishing several important results of this article. We already noticed, in Section~\ref{se:lrs}, that condition {\bf{a)}} 
    of Theorem~\ref{Thm:Main} is essential. Now, we illustrate with an example that all of the $\gcd(\tau_i,\tau_j)'s$ cannot be too large. 
    For example, let $r=2$ and $g$ be a generator of $\rF_{\ell^2}^*.$
    Then, consider the sequence
  \[
    s_n=\tr{g^{n(\ell^2+1)/2} - g^n},
  \]  
    with characteristic polynomial   $(x-g)(x-g^{\ell})(x-g^{(\ell^2+1)/2})(x-g^{\ell(\ell^2+1)/2})$. Note that 
    \[\tau_2=\ord{g}=\ell^2-1~\mathrm{and}~\tau_1 = \ord{g^{(\ell^2+1)/2}}=\tfrac{\ell^2-1}{\gcd(\ell^2-1,(\ell^2+1)/2)}.
  \]
    It is easy to see that $\gcd(\ell^2-1,(\ell^2+1)/2)=1$ 
    or $2,$
    therefore
  \[
    \gcd(\tau_1,\tau_2)=\begin{cases}
      \ell^2-1 & \textrm{if }\, \gcd(\ell^2-1,(\ell^2+1)/2)=1\\
      (\ell^2-1)/2 & \textrm{if }\, \gcd(\ell^2-1,(\ell^2+1)/2)=2
    \end{cases}.
  \]     
     Then, one can show that
  \begin{align*}
     \sum_{n =1}^{\ell^2 -1} \el{s_n} &= \sum_{n=1}^{\ell^2 -1} \el{\tr{g^{n(\ell^2+1)/2} - g^n}} \\
          &= \frac{\ell^2-1}{2} + \sum_{n=1}^{(\ell^2 -1)/2} \el{\tr{-2g^{2n+1}}}.
  \end{align*}
  Noting that $\ord{g^2}= (\ell^2-1)/2$ and using \cite[Theorem 3.2]{Kowalski}, we have
  \[
    \left| \sum_{n=1}^{(\ell^2 -1)/2} \el{\tr{-2g^{2n+1}}}\right| \le
    \max_{\xi \in \rF_{\ell^2}^*} \left| \sum_{n=1}^{(\ell^2 -1)/2} \el{\tr{\xi g^{2n}}}\right| = O(\ell).
  \]    
    Therefore, the linear recurrence sequence 
     $\{s_n\}$ satisfies
  \[
    \sum_{n =1}^{\ell^2 -1} \el{s_n} = \frac{\ell^2-1}{2} + O(\ell).
  \]     
  We now need to discuss some necessary background. Let $K$ be a finite field of 
    characteristic $p$ and $F$ be an extension of $K$ with $[F:K]=r.$ The {\it trace} 
    function  $\Tr{F/K} : F \to K$ is defined by
  \[
    \Tr{F/K}(z)= z + z^p + \cdots +z^{p^{r-1}}, \qquad z\in F.
  \]  
    The following properties of $\Tr{F/K}(z)$ are well known.
       \begin{align}
            \Tr{F/K}(az + w) & = a\Tr{F/K}(z) + \Tr{F/K}(w), \quad \textrm{for all } a \in K, \, z,w\in F. \label{eq:TraceLinear} \\
            \Tr{F/K}(a) &=ra, \quad \textrm{for any} \quad a \in K. \label{eq:TraceScalar} \\
            \Tr{F/K}(z^p) & = \Tr{F/K}(z), \quad \textrm{for any} \quad z \in F. \label{eq:TracePowers}
    \end{align}
      Throughout this section, $F=\rF_q$, $K=\rF_p$ with $q=p^r$ and we will simply write $\tr{z}$ 
      instead $\Tr{F/K}(z)$.

   Let $\{s_n\}$ be a linear recurrence sequence of order $r\ge 1$ in $\rF_p$ with 
    characteristic polynomial $\omega(x)$ in $\rF_p[x].$ 
    It is well known that $n^{th}$-term can be written in terms of the roots 
    of the characteristic polynomial, see Theorem 6.21 in~\cite{Lidl-Niederreiter}. Therefore, if  the roots $\alpha_0, \ldots, \alpha_{r-1}$
    of $\omega(x)$ are all distinct in its splitting field, then 
  \begin{equation}\label{eq:RootsPowers}
    s_n= \sum_{i=0}^{r-1} \beta_i \alpha_i^n, \quad \textrm{for }\, n=0,1,2,\ldots,
  \end{equation}
    where $\beta_0,\ldots,\beta_{r-1}$ are uniquely determined by initial values $s_0, \ldots, s_{r-1},$
    and belong to the splitting field of $\omega(x)$ over $\rF_p.$

    If the characteristic polynomial $\omega(x)$ is irreducible and $\alpha$ is a root, then its 
    $r$ distinct conjugates are
  \[
     \alpha, \alpha^{p}, \ldots, \alpha^{p^{r-2}}, \alpha^{p^{r-1}}.
  \]        
    Hence, the coefficients $s_n$ are given by 
  \begin{equation*}\label{eq:irreducibleS_n}
      s_n=\sum_{i=0}^{r-1}\beta_i \alpha^{p^in}, \qquad n=0,1,2,3,\ldots \, .
  \end{equation*} 
    
   One of our main tools is the bound for Gauss sum in finite fields given by
   Bourgain and Chang~\cite[Theorem 2]{BourgainChang2006}. This will be required to prove Theorem~\ref{Thm:Main}. 
       Assume that for a given $\alpha \in \rF_{q}$ 
     such that $t=\ord{\alpha}$ satisfies
   \begin{equation}\label{eq:BourgainChang}
      t > p^{\varepsilon} \quad \textrm{and} \quad \max_{\substack{1\le d < r \\ d |r}} 
      \gcd (t, p^d - 1) < t p^{-\varepsilon}.
   \end{equation}
      Then there exists a $\delta=\delta(\varepsilon)>0$ such that for any nontrivial additive character $\psi$ of $\rF_q,$ we have
   \[
      \left|\sum_{n \le t} \psi(\alpha^n) \right| \le t p^{-\delta}.
   \]
     Here, we note that the second assumption in \eqref{eq:BourgainChang} implies the first one whenever $r\ge 2$. 

\subsection{Proof of Theorem~\ref{Thm:Main}}
    We proceed by induction over $\nu.$ Before that,
    following properties \eqr{TraceLinear} and \eqr{TraceScalar} of trace function we get
    {\small
  \begin{align*}
     s_{n}& =\tr{r^{-1}s_{n}}=r^{-1}\tr{\sum_{i=1}^{\nu} (\beta_{i,0}\alpha_i^n + 
      \cdots + \beta_{i,t_i-1}\alpha_i^{p^{t_i-1}n})} 
       =r^{-1} \sum_{i=1}^{\nu}  \sum_{j=0}^{t_i-1} \tr{ \beta_{i,j} \alpha_i^{p^{j}n}}.
  \end{align*}     }
    Set $r=[\rF_p(\alpha_1, \ldots, \alpha_{\nu}):\rF_p],$ then $z^{p^r}=z$ for any 
    $z \in \rF_p(\alpha_1, \ldots, \alpha_{\nu}),$ in particular $\tr{z^{p^u}}=\tr{z}.$
    Then for each pair $(i,j),$ raising each argument $\beta_{i,j}\alpha_i^{p^jn}$
    to the power $p^{r-j}$ 
  \[
    \tr{\beta_{i,j}\alpha_i^{p^jn}}=\tr{\beta_{i,j}^{p^{r-j}}\alpha_i^{p^jn \cdot p^{r-j}}}=
    \tr{\beta_{i,j}^{p^{r-j}}\alpha_i^{p^rn}} = \tr{\beta_{i,j}^{p^{r-j}}\alpha_i^{n}}.
  \]
    This implies that 
  \begin{align}\label{eq:s_nTrace}
    s_{n} & = r^{-1} \sum_{i=1}^{\nu}  \sum_{j=0}^{t_i-1} \tr{ \beta_{i,j}^{p^{r-j}} \alpha_i^{n}}
            = r^{-1} \sum_{i=1}^{\nu} 
            \tr{ \left(\sum_{j=0}^{t_i-1} \beta_{i,j}^{p^{r-i}}\right)\alpha_i^{n}} \nonumber \\
          & = 
             \tr{\gamma_1 \alpha_1^n }+ \cdots + \tr{\gamma_{\nu} \alpha_{\nu}^n},
         \end{align}
     where $\gamma_i = r^{-1}\sum_{j=0}^{t_i-1} \beta_{i,j}^{p^{r-i}},$ for each $1 \le i \le \nu.$ 
     
    The case $\nu = 1$ follows by Bourgain and Chang \cite[Theorem 2]{BourgainChang2006}. 
    We shall now proceed inductively, and $\nu=2$ will be the base case. We start by denoting
    $h=\mathrm{gcd}(\tau_1,\tau_2).$
    It is clear that $\textrm{lcm}(\tau_1, \tau_2)=\tau_1 \tau_2/h$ is a period of $s_n,$ then 
  \[
     \left|\sum_{n\le \tau} \ep{\xi s_n} \right|=
    \frac{\tau}{\tau_1 \tau_2/h}\left|\sum_{n\le \tfrac{\tau_1 \tau_2}{h}} \ep{\xi s_n}\right|.
  \]    
    Hence, it is enough to prove that
  \[
    \left| \sum_{n\le \tfrac{\tau_1 \tau_2}{h}} \ep{\xi s_n} \right| \le \frac{\tau_1 \tau_2}{h} p^{-\delta}, \quad \textrm{with }\,
    (\xi,p)=1,
  \]    
    for some $\delta = \delta(\varepsilon)>0.$ We have
  \begin{align}\label{eq:nu=2}
    \left| \sum_{n\le \tfrac{\tau_1 \tau_2}{h}} \ep{\xi s_n} \right| & =
    \left| \sum_{u=0}^{h-1}\sum_{n\le \tfrac{\tau_1 \tau_2}{h^2}} \ep{\xi s_{nh + u}} \right| \le
    \sum_{u=0}^{h-1} \left| \sum_{n\le \tfrac{\tau_1 \tau_2}{h^2}} \ep{\xi s_{nh + u}} \right|\nonumber \\
      & \le h \times  
    \max_{0 \le u \le h-1}\left| \sum_{n\le \tau_1 \tau_2/h^2} \ep{\xi s_{nh + u}} \right|.
  \end{align}
 Let  $(n_1,n_2)$ be a tuple with $n_i\leq \frac{\tau_i}{h}.$ Since $\gcd(\tfrac{\tau_1}{h},\tfrac{\tau_2}{h})=1$, 
 by Chinese remainder theorem, 
     there exist integers $m_1,m_2$ with $\gcd(m_1, \tfrac{\tau_1}{h})= \gcd(m_2,\tfrac{\tau_2}{h_2})=1, $
     such that
  \begin{multline}\label{eq:Chinese}
    \left|\left\{ n \!\!\!\pmod{ \tfrac{\tau_1 \tau_2}{h^2}} \,:\, 
           1\le n \le \frac{\tau_1 \tau_2}{h^2} \right\}\right| = \\
        =\left|\left\{ n_1 m_1 \tfrac{\tau_2}{h}  +n_{2} m_{2} \tfrac{\tau_1}{h}
        \! \! \!\pmod{ \tfrac{\tau_1 \tau_2}{ h^2 }} \,:\, 
        1\le n_i \le \frac{\tau_i}{h} \right\}\right|.
  \end{multline} 
    Moreover, the pair $(m_1, m_2 )$ has the following property: 
    given $(n_1, n_{2}),$ with $1\le n_i\le \tau_i/h,$ then
    $n =  n_1 m_1 \tfrac{\tau_2}{h} +n_{2} m_{2} \tfrac{\tau_1}{h}$
    satisfies 
  \begin{equation*}
    n \equiv n_1 \!\!\! \pmod{\tfrac{\tau_1}{h}} \; \textrm{and }\,
    n \equiv n_2 \!\!\! \pmod{\tfrac{\tau_2}{h} },
  \end{equation*}
    and $n$ is unique modulo $\tfrac{\tau_1\tau_2}{h^2}.$
    Since $\tfrac{\tau_1}{h}= \ord{\alpha_1^h}$ and $\tfrac{\tau_2}{h}= \ord{\alpha_2^h}$
    then 
  \begin{equation}\label{eq:alphapower}
    \alpha_i^{hn} = 
    \alpha_i^{h(n_1 m_1 \tfrac{\tau_2}{h}  +n_{2} m_{2} \tfrac{\tau_{1}}{h})}
     = \alpha_i^{hn_i}, \quad 1 \le i \le 2.
  \end{equation} 
    Combining  \eqref{eq:Chinese} and \eqref{eq:alphapower} we have
 {\small    
  \begin{align}
     \left|\sum_{n\le \tfrac{\tau_1 \tau_2}{h^2}}\ep{\xi s_{nh+u}} \right| &=
     \left| \sum_{n_1\le \tfrac{\tau_1 }{h}}\ep{\tr{\xi \gamma_1 \alpha_1^{n_1h + u} }} \right| 
                 \times 
             \left |\sum_{n_2\le \tfrac{ \tau_2}{h}}
             \ep{ \tr{\xi \gamma_{2} \alpha_{2}^{n_{2}h+u}}}
             \right| \nonumber \\
     & =\left| \sum_{n_1\le \tfrac{\tau_1 }{h}}\ep{\tr{\gamma'_1 \alpha_1^{n_1h} }} \right| 
                 \times 
             \left |\sum_{n_2\le \tfrac{ \tau_2}{h}}
             \ep{ \tr{\gamma'_{2} \alpha_{2}^{n_{2}h}}}
             \right|, \label{eq:nu=2split}
  \end{align}}
     with $\gamma'_1=\xi \gamma_1 \alpha_1^{u}, \gamma'_2=\xi\gamma_2 \alpha_2^{u}$ in $\rF_p(\alpha_1, \alpha_2).$ 
     Since $\{s_n\}$ is a nonzero sequence then $\gamma'_i \neq 0,$ at least for
     some $1\le i \le 2.$ We may assume that all of them are nonzero. Each $\ep{ \tr{\xi\gamma'_{i} z}}$ corresponds to a nontrivial additive character, say $\psi_i(z),$ in $\rF_p(\alpha_i)=\rF_{p^r}.$ In order to satisfy condition~\eqref{eq:BourgainChang} we first 
     recall assumptions
    $h < p^{\varepsilon'},$  $\varepsilon > \varepsilon'>0$ and $\max_{\substack{ d < r \\ d | r }} \gcd(\tau_i, p^d-1) <  \tau_i p^{-\varepsilon}$
    for some $i \in \{1,2\}.$
    Then, 
    for any $d|r$ with $1\le d < r$ and some $i=1,2 $ we have
  \[
    \gcd\left(\tfrac{\tau_i}{h} , p^d - 1 \right)  \le \gcd(\alpha_i, p^d -1) < \tau_i p^{-\varepsilon} <
    \frac{\tau_i}{h}p^{-(\varepsilon-\varepsilon')}.
  \]    
    Therefore, by Bourgain and Chang~\cite[Theorem 2]{BourgainChang2006}  it follows that
     \[
    \left| \sum_{n_i \le \tau_i/h}\ep{\tr{\gamma'_i \alpha_i^{n_ih} }} \right| 
    =\left| \sum_{n_i \le \tau_i/h}\psi({ \alpha_i^{n_ih} }) \right| 
    \leq \frac{\tau_i}{h}p^{-\delta}, \quad \textrm{for some } \; 1\le i \le2.
  \]
    Thus, combining above equation with \eqref{eq:nu=2} and \eqref{eq:nu=2split} we get
  \[
    \max_{\xi \in \rF_p^*} \left|\sum_{n \le \frac{\tau_1\tau_2}{h}}\ep{\xi s_{nh}} \right| \leq h \times 
         \frac{\tau_1 \tau_2}{h^2}p^{-\delta}=\frac{\tau_1 \tau_2}{h}p^{-\delta}.
  \]
    This conclude the case $\nu=2.$ Now we proceed by induction over $\nu,$ and assume Theorem \ref{Thm:Main} to be true up to
    $\nu-1.$ We follow the idea due to Garaev~\cite[Section 4.4]{Garaev2010}. Considering \eqref{eq:s_nTrace} and periodicity, for any
    $t\ge 1$ we get
  \begin{align*}
    \tau \left| \sum_{n\le \tau} \ep{\xi s_n}\right|^{2t} & = \sum_{m\le \tau}\left| \sum_{n\le \tau} \ep{\xi s_{m+n}}\right|^{2t} \\
        &= \sum_{m\le \tau}\left|\sum_{n \le T}\ep{\xi (\tr{\gamma_1 \alpha_1^{m+n} }+ \cdots + \tr{\gamma_{\nu}
          \alpha_{\nu}^{m+n}})} \right|^{2t} \\
         & \le \sum_{n_1\le \tau} \cdots \sum_{n_{2t}\le \tau} \left| \sum_{m\le \tau} 
           \ep{\xi \sum_{i=1}^{\nu} \tr{\gamma_i\alpha_i^{m}\left(\alpha_i^{n_1}
             + \cdots - \alpha_i^{n_{2t}}\right)}}\right|.
  \end{align*}
     Raising to the power $2t,$ and applying Cauchy–Schwarz, we have
  \begin{align*}
     \tau^{2t}\left| \sum_{n\le \tau} \ep{\xi s_n}\right|^{4t^2}
          &\le {\tau^{2t(2t-1)}} \sum_{n_1\le \tau} \cdots \sum_{n_{2t}\le \tau} 
            \left| \sum_{m\le \tau} \ep{\xi \sum_{i=1}^{\nu} \tr{\gamma_i\alpha_i^{m}\left(\alpha_i^{n_1}
             + \cdots - \alpha_i^{n_{2t}}\right)}}\right|^{2t}.
            \end{align*}
     Given $(\lambda_1 ,\cdots, \lambda_{\nu}) \in \rF_q^{\nu},$  
     let $J_t(\lambda_1 ,\cdots, \lambda_{\nu})$ denote 
     the number of solutions of the system 
  \begin{equation*}\label{system:SumPowersAlpha}
     \left\{
     \begin{matrix}
             \alpha_1^{n_1}  +   \cdots + \alpha_{1}^{n_t} &=& \alpha_{1}^{n_{t+1}}  +   \cdots + \alpha_1^{n_{2t}} + \lambda_1 \\
                 \vdots       \qquad    \vdots          & &  \vdots          \qquad      \vdots   \qquad \vdots  \\ 
             \alpha_{\nu}^{n_1}+  \cdots + \alpha_{\nu}^{n_t} &=& \alpha_{\nu}^{n_{t+1}}  +   \cdots + \alpha_{\nu}^{n_{2t}} + \lambda_{\nu}
    \end{matrix} 
    \right.      
  \end{equation*}
    with $1 \le n_1, \cdots, n_{2t} \le \tau.$ Therefore,
  {\small    
  \begin{align}\label{ineq:s_n-powers}
          \left| \sum_{n\le \tau} \ep{\xi s_n}\right|^{4t^2} &\le \tau^{4t^2-4t} \sum_{\lambda_1 \in \rF_q} \cdots \sum_{\lambda_{\nu} \in \rF_q}
            J_t(\lambda_1 ,\cdots, \lambda_{\nu})  
             \left| \sum_{m\le \tau}
            \ep{ \xi\sum_{i=1}^{\nu}\tr{\gamma_i \lambda_i \alpha_i^{m}}}\right|^{2t}.
  \end{align}}
      Note that writing $J_{\nu}(\lambda_1 \cdots, \lambda_{\nu})$ in terms of character sums it follows that
  \begin{align*}
      J_{t}(\lambda_1 \cdots, \lambda_{\nu})  &= 
              \frac{1}{q^{\nu}}\sum_{\psi_1 \in \rwidehat{\rF}_{q}} \cdots \sum_{\psi_{\nu} \in \rwidehat{\rF}_{q}}
              \left|\sum_{n\le \tau} \psi_1(\alpha_1^n) \cdots \psi_{\nu}(\alpha_{\nu}^n)  \right|^{2t} \times  \psi_1(\lambda_1) \cdots \psi_{\nu}(\lambda_{\nu}) \nonumber \\
            & \le \frac{1}{q^{\nu}}\sum_{\psi_1 \in \rwidehat{\rF}_{q}} \cdots \sum_{\psi_{\nu} \in \rwidehat{\rF}_{q}}
              \left|\sum_{n\le \tau} \psi_1(\alpha_1^n) \cdots \psi_{\nu}(\alpha_{\nu}^n)  \right|^{2t} \nonumber \\
            &  \le  J_{t}(0, \ldots, 0) =:J_{t,\nu}.
  \end{align*}
     In particular, we note that 
     $J_{t,\nu} \le J_{t,\nu-1}.$ From~\eqref{ineq:s_n-powers} it follows that 
     {\small
  \begin{align*}
          \left|\sum_{n\le \tau} \ep{\xi s_n}\right|^{4t^2} \le \tau^{4t^2-4t} J_{t,\nu} 
             \sum_{m_1\le \tau} \cdots  \sum_{m_{2t}\le \tau} 
           \sum_{\lambda_1 \in \rF_q} \cdots \sum_{\lambda_{\nu} \in \rF_q}
             \ep{\sum_{i=1}^{\nu}\tr{\xi\beta_i\lambda_i (\alpha_i^{m_1} + \cdots - \alpha_i^{m_{2t}})}}
  \end{align*}}
    Note that  $a \gamma \lambda,$ with $a\gamma\neq0,$ runs over $\lambda \in \rF_q$, 
    then $\ep{\tr{a \theta \lambda z}}$    
    runs through all additive characters $\psi$ in $\rwidehat{\rF}_{q},$ evaluated at $z.$ Then the above expression can be written 
    as 
  {  
  \begin{align}
     \left| \sum_{n\le \tau} \ep{\xi s_n}\right|^{4t^2}  &\le \tau^{4t^2-4t} J_{t,\nu}  \sum_{m_1\le \tau} \cdots  \sum_{m_{2t}\le \tau}  
                \prod_{i=1}^{\nu}
              \left(\sum_{\psi_i \in \rwidehat{\rF}_q}  \psi_i \left( \alpha_i^{m_1} + \cdots - \alpha_i^{m_{2t}}\right)\right)
               \nonumber  \\
          & \le \tau^{4t^2-4t} q^{\nu} J^2_{t,\nu} \le \tau^{4t^2-4t} q^{\nu} J^2_{t,\nu-1}. \label{ineq:S(a)^(4k^2)}
  \end{align}}
    We now require an estimate for $J_{t,\nu-1}.$ Writing it as the sum of characters 
    {\small
      \begin{align}\label{ineq:nu} 
    J_{t,\nu-1} &= \frac{1}{q^{\nu-1}}\sum_{\lambda_1 \in {\rF_{q}}} \cdots \sum_{\lambda_{\nu-1} \in {\rF_{q}}}
              \left|\sum_{m\le \tau} \ep{\tr{\lambda_1 \alpha_1^m + \cdots + \lambda_{\nu-1}\alpha_{\nu-1}^m}} 
              \right|^{2t} \nonumber \\
           & = \frac{\tau^{2t}}{q^{\nu-1}} +
              O\left( \left(
              \max_{\substack{(\lambda_1, \ldots, \lambda_{\nu-1})\in \rF_q^{\nu-1} \\ (\lambda_1, \ldots, \lambda_{\nu-1}) \neq 0}}
              \left|\sum_{m\le \tau} \ep{\tr{\lambda_1 \alpha_1^m + \cdots + \lambda_{\nu-1}\alpha_{\nu-1}^m}} 
              \right|
              \right)^{2t}\right). 
  \end{align}}
     Finally, we note that $s'_m=\tr{\lambda_1 \alpha_1^m + \cdots + \lambda_{\nu-1}\alpha_{\nu-1}^m}$ defines a linear recurrence sequence 
     with period $\tau'$ dividing $\tau,$ which in particular satisfies induction hypothesis. Therefore
  \[
    \left|\sum_{m\le \tau} \ep{\tr{\lambda_1 \alpha_1^m + \cdots + \lambda_{\nu-1}\alpha_{\nu-1}^m}} 
              \right| \le \tau p^{-\delta'},
  \]     
    for some $\delta'=\delta'(\varepsilon)>0.$ Now, taking $t > d(\nu-1)/ 2\delta'$ (where $d=[\rF_q:\rF_p]$) and combining with
    \eqref{ineq:nu} we get
  \[
    J_{t,\nu-1} \ll \frac{\tau^{2t}}{q^{\nu-1}}.
  \]    
    We conclude the proof combining the above estimate with \eqref{ineq:S(a)^(4k^2)} to get
  \[
    \max_{\xi \in \rF_p^*}\left| \sum_{n\le \tau} \ep{\xi s_n}\right| \le \tau p^{-\delta},
  \]    
    with $\delta=-\tfrac{d(\nu-2)}{4t^2}.$\footnote{To get a non trivial estimate, we must have a non zero $\delta.$ This is true when $\nu>2.$ Hence our induction step starts from $\nu=2.$ }

The following is an immediate corollary of this theorem which will be quite handy in establishing several results in Section~\ref{MF} and in Section~\ref{se:waring}.
\begin{corollary}\label{coro:ThmMain}
    Suppose that  $\{ s_n\}$ is a nonzero linear recurrence sequence 
    of order $r\ge 2$  such that its characteristic polynomial $\omega(x)$  is
    irreducible in $\rF_p[x].$  
    If its period $\tau$ satisfies
    
   \begin{equation*}
      \max_{\substack{ d< r \\ d| r}}  \gcd({\tau}, p^d-1) 
                < {\tau} \, p^{-\varepsilon},
   \end{equation*} 
    then there exists a $\delta=\delta(\varepsilon)>0$ such that
   \begin{equation*}
     \max_{\xi \in \rF_p^*} \left|\sum_{n \le \tau}\ep{\xi s_n} \right| \le  \tau { p^{-\delta}}.
   \end{equation*}    
\end{corollary}

\section{Exponential sums for  modular forms}\label{MF}
 In this section, we study the effect of linear recurrence sequence and Theorem~\ref{Thm:Main} 
 in the behaviour of the exponential sums attached to certain Fourier coefficients of modular forms. 
 As a consequence, we obtain interesting results which have been summarized earlier in the form of 
 Theorem~\ref{thm:main1-1} and Theorem~\ref{thm:main1-2}.

\subsection{Order of the roots of the characteristic polynomial}
In the case of normalized eigenforms, the sequence $\{a(p^n)\}$ defines a linear recurrence sequence of order two when $p\nmid N,$ and otherwise it is of order one. This is one of the tools for Theorem~\ref{thm:main1-1}. However, we do not need to assume that the form is normalized because 
the normalizing factor is in $\mathbb{Q},$ and we can realize that to be an element of $\mathbb{F}_{\ell}^{*}$ for any 
large enough prime $\ell.$ Before going into the proof of this theorem, we develop a tool  which will be quiet useful throughout.
We state it in the form of following lemma.

\begin{lemma}\label{Igor1} 
    Let $\omega(x)=x^2 + ax +b  \in \mathbb{Z}[x]$ be a quadratic polynomial with $b \neq 0$ and let $\alpha, \beta$ be its roots such that none of $\alpha, \beta$ or $\alpha\beta^{-1}$ is a root of unity. For any prime $\ell,$ let
    $\alpha_{\ell},$ $\beta_{\ell}$ be its roots in the splitting field of $\omega(x)$ over $\rF_{\ell}.$
    Then, given $0< \varepsilon< 1/2,$  for $\pi(y) + O(y^{2\varepsilon})$ many
    primes $\ell \leq y,$  we have 
  \[
    \ord{\alpha_{\ell}}> \ell^{\varepsilon}, \qquad \ord{\beta_{\ell}} > \ell^{\varepsilon}
    \quad \textrm{and} \quad \ord{(\alpha_{\ell} \beta_{\ell}^{-1}})> \ell^{\varepsilon}.
  \]    
\end{lemma}
\begin{proof}
    For notational convenience, throughout the proof we will simply write $\omega,G_T$ instead of $\omega(x), G_T(x)$ respectively, and $\omega \hspace{-0.1cm}\pmod \ell, G_T\hspace{-0.1cm}\pmod \ell$ for the reduced polynomials $\omega(x) \hspace{-0.1cm}\pmod \ell, G_T(x)\hspace{-0.1cm}\pmod \ell$ respectively. It is clear that $\omega \hspace{-0.1cm} \pmod {\ell}$ has distinct roots for all but finitely many primes $\ell,$ since $a^2 - 4b\neq 0.$ 
    For any such prime $\ell$, let $\alpha_{\ell}$ and $\beta_{\ell}$ be the distinct roots in its splitting field. Given a large positive parameter $T,$ consider the polynomial  
  \[
    G_T(x)= \prod_{t \le T} (x^t - 1 )(x^{2t}-b^t) \in \mathbb{Z}[x].
  \]
    We first consider the resultant $\textrm{Res}(\omega,G_T)$, and note that
  \[
    \textrm{Res}(\omega,G_T) \hspace{-0.2cm}\pmod \ell= \prod_{1\le i \le 3T}(\alpha_{\ell} -\mu_i)(\beta_{\ell} -\mu_i),
  \]     
    where $\mu_i$ are the roots of $G_T$ in its splitting field over $\rF_{\ell}.$ 
    
    In particular $\textrm{Res}(\omega,G_T)\equiv 0 \hspace{-0.1cm} \pmod{\ell}$ if and only if
    $\omega \hspace{-0.1cm} \pmod \ell$ and $G_T \pmod \ell$ have common roots in some finite extension of $\mathbb{F}_{\ell}.$ 
    Additionally, since $\alpha_{\ell} \beta_{\ell} = b,$ it follows that $\ord{(\alpha_{\ell}\beta_{\ell}^{-1})} \le T$ if and only if 
    $\alpha_{\ell}^{2t} - b^t =0$ (or $\beta_{\ell}^{2t} - b^t =0$), for some $t \le T.$ 
    Therefore, $\alpha_{\ell}$ or $\beta_{\ell}$ are common roots of $\omega \hspace{-0.1cm} \pmod \ell$ and $G_{T} \hspace{-0.1cm} \pmod \ell$ if
    $\ord{\alpha_{\ell}}, \ord{\beta_{\ell}}$ or $\ord{(\alpha_{\ell}\beta_{\ell}^{-1})}$ are less than $T.$ Now, the Sylvester matrix of $\omega$ and $G_T$ 
    is a square matrix of order $2 + \deg G_T\ll T^2,$ and entries bounded 
    by an absolute constant $M$ (which depends on $a,b$ and not on $\ell$ or the parameter $T$). 
    Then, the determinant
  \[
     \textrm{Res}(\omega,G_T)\le T^2! \times M^{T^2} \ll M^{2T^2 \log T}.
  \]    
   Note that $\text{Res}(\omega,G_T)$ is zero if and only if $\alpha^t=1, \beta^{t}=1$ or $(\alpha\beta^{-1})^t=1$ for some $t\leq T,$ 
   which following our assumption can not happen. In particular, the resultant has at most $O\left(T^2\right)$ many distinct prime divisors. This shows that 
       \[
       |\{\ell \mid \ord{\alpha_{\ell}}\leq T \quad \textrm{or }\,  \ord{\beta_{\ell}}\leq T \quad \textrm{or }\, \ord{\alpha_{\ell}\beta_{\ell}}^{-1}\leq T\}|=O(T^2).
     \]
    Choosing $T=y^{\varepsilon}$
    the number of primes $\ell \le y$ such that 
  \[  
    \ord{\alpha_{\ell}} \le \ell^{\varepsilon} \quad \textrm{or} \quad \ord{\beta_{\ell}} \le \ell^{\varepsilon}
    \quad \textrm{or} \quad \ord{(\alpha_{\ell} \beta_{\ell}^{-1}}) \le \ell^{\varepsilon}
  \]    
    is $O\left(y^{2\varepsilon}\right).$
    \end{proof}

Let us now proceed to prove the main result of this section.
\subsection{Proof of Theorem~\ref{thm:main1-1}} If $p\mid N,$ then we only need to consider 
\begin{equation}
    \max_{\xi \in \rF_{\ell}^*} \left|\sum_{n\le \tau}\el{\xi a(p)^n }\right|. 
  \end{equation}
  Note that, if $p\notin \mathcal{P}$ then $a(p)=0,$ and the problem is trivial in this case because, we have $\tau=1.$ On the other hand, if $p\in \mathcal{P},$ then for any large enough prime $\ell, \tau$ is simply the order of $a(p)\hspace{-0.1cm}\pmod \ell$ in $\mathbb{F}_{\ell}^{*}.$ Due to Lemma~\ref{Igor1}, we may assume that $\tau>p^{\varepsilon}$ holds for $\pi(y)+O(y^{2\varepsilon})$ many primes $\ell<y.$ Hence, this case is settled down by~\cite[Theorem 6]{BGK2006}.

    Let us now consider the case $p\nmid N$. The characteristic polynomial of \eqref{identity:primepowers} is
\begin{equation}\label{eq:associate} 
    \omega(x)=x^2- a(p)x + p^{k-1}, 
  \end{equation}    
     and has discriminant $a^2(p)-4p^{k-1}.$ We note that in our case the discriminant does not vanish,
     otherwise $|a(p)|=2p^{(k-1)/2}$ is absurd, with $a(p)$ being integer and $p^{(k-1)/2}$ irrational.
     Now, let $\rP$ be the set of all primes. We divide the proof for primes $p\in \cP$ and $p\in \rP \setminus \cP.$ Since $a^2(p)-4 p^{k-1}\neq 0,$ for any $p\in \cP,$ we write $a^2(p)-4 p^{k-1}=u^2 D_p,$  
    with $D_p<0$ square-free and $u\neq0.$ We now split the cases according to $D_p \hspace{-0.1cm}\pmod \ell$  
    is quadratic residue, zero or non quadratic residue modulo $\ell.$ Set 
  \[
    \rP = \rP_{0} \cup \rP_{1} \cup \rP_{-1}, \quad \textrm{where } 
    \rP_{\nu}=\left\{\ell \in \rP  \;:\; \left(\frac{D_p}{\ell}\right)=  \nu \right\}.
  \] 
    For $\nu = 0, 1, -1$, we also define 
  \[
     \rP_{\nu}(x)=\rP_{\nu}\cap [1,x],\quad \pi_{\nu}(x)=\left|\rP_{\nu}(x)\right| \quad \textrm{and} \quad  \kappa_{\nu}=\lim_{x \to \infty} \frac{\pi_{\nu}(x)}{\pi(x)}.
  \]    
    It is clear that $\pi_{\nu}(x)=\pi(x)(\kappa_{\nu}+o(1)),$ and $\kappa_0+\kappa_1 + \kappa_{-1}=1.$ 

    Note that for a given prime $p,$ the associated polynomial $\omega(x) \pmod \ell$ has a single root in $\rF_{\ell}$ if and only if
    $u^2D_p \equiv 0 \pmod \ell.$ Since such equation has finitely many solutions for $\ell,$ we get 
    $\kappa_0=0.$
    On the other hand, Chebotarev's density theorem implies that the uniform distribution of primes $\ell$
    such that $\omega(x) \hspace{-0.1cm}\pmod \ell$ is irreducible or has distinct roots in $\rF_{\ell}$. Equivalently, the primes $\ell$
    satisfying $\left(\tfrac{D_p}{\ell}\right)=\pm 1$ are distributed in the same proportion, 
    therefore $\kappa_{-1}=\kappa_1=1/2.$ 
    We now turn to establish nontrivial exponential sums for $\{a(p^n)\}\pmod \ell$ with
    $\ell \in \rP_{\nu}$ for $\nu=\pm 1.$ 

\subsection*{Case 1. $\ell \in \rP_{-1}$:}  we want to show that the inequality~~\eqref{eq:main1} is satisfied 
    by $\frac{\pi(y)}{2}+O(y^{2\varepsilon})$ many primes 
    $\ell \leq y$ in $\rP_{-1}.$ In this case the associated polynomial~\eqref{eq:associate}
    is irreducible modulo $\ell,$ then the idea is to employ Corollary~\ref{coro:ThmMain}.
    Let $\alpha$ and  $\beta =\alpha^{\ell}$  be the conjugate roots of \eqref{eq:associate} in 
    its splitting field $\rF_{\ell}(\alpha).$ For a given $\varepsilon>0,$ from Lemma \ref{Igor1} it follows that for $\pi(y)+O(y^{2\varepsilon})$ many primes $\ell \leq y,$ the following inequalities
\begin{equation}\label{ineq:RootsOrders}
     \ord{\alpha^{\ell}} = \ord{\alpha}  > \ell^{\varepsilon} \quad \textrm{and} \quad
     \ord{\alpha \beta^{-1}}=\ord{\alpha^{1-\ell}} > \ell^{\varepsilon}
   \end{equation}
    hold. Combining the identity 
  \[    
    \ord{\alpha^{\ell - 1}}=\frac{\ord{\alpha}}{\gcd(\ord{\alpha},\ell-1)}
  \]    
    with the second inequality of \eqref{ineq:RootsOrders}, we get
  \[
    \gcd(\ord{\alpha},\ell-1)=\frac{\ord{\alpha}}{\ord{\alpha^{\ell - 1}}}=
     \frac{\ord{\alpha}}{\ord{\alpha^{1 - \ell}}} 
     < (\ord{\alpha}) \ell^{-\varepsilon}.
  \]    
    Now applying Corollary~\ref{coro:ThmMain} we complete the proof of this case.

\subsection*{Case 2. $\ell \in \rP_{1}$:} let $\alpha,\beta$ be the roots of $\omega(x) \pmod{\ell}$ inside $\mathbb{F}^*_{\ell}.$ From  \eqref{eq:RootsPowers} 
    it follows that for $n\geq 0,$ $ a(p^n) \equiv  c \alpha^n + d \beta^n \pmod{\ell}, $
    for some constants $c,d$ in $\rF_{\ell},$ with $(\alpha, \beta)\neq (0,0).$ 
    It is clear that $\ell-1$ is a period of the sequence $a(p^n) \pmod{\ell},$ and hence $\tau $ divides $\ell-1.$
    We have
  \[
    \sum_{n\le \tau} \el{\xi a(p^n)} = \frac{\tau}{\ell-1} \sum_{n\le \ell-1} \el{\xi a(p^n)}
       = \frac{\tau}{\ell-1} \sum_{n\le \ell-1} \el{\xi (c\alpha^n + b \beta^n)}.
  \]    
    From Lemma~\ref{Igor1}, there is a subset of $\rP_{1}$ with $\frac{\pi(y)}{2}+O(y^{2\varepsilon})$ many primes $\ell \leq y$ such  
    that $\ord \alpha, \ord \beta$ and $\ord (\alpha \beta^{-1})$ are bigger 
    than $\ell^{\varepsilon}.$ It follows from~\cite[Corollary]{Bourgain2005a} that there exists a $\delta=\delta(\varepsilon)>0$ such that
  \[
    \max_{\substack{ (c,d)\in \rF_{\ell} \times \rF_{\ell} \\ (c,d)\neq (0,0) }}
    \left|\sum_{{ n \le \ell-1}}\el{c \alpha^n + d \beta^n } \right| \le \ell^{1-\delta}.
  \]     
    Hence, \emph{(i)} of Theorem~\ref{thm:main1-1} holds. Now, assume that $p$ belongs to the exceptional set $\rP \setminus \cP,$ that is  $a(p^u)=0$ 
    for some $u\ge 1.$ We consider $u=u(p)$ to be the least such integer. Since the discriminant is nonzero (the roots
    $\alpha$ and $\beta$ of \eqref{eq:associate} are distinct),  we get 
  \begin{equation*}
    a(p^u)=\frac{\alpha^{u+1}-\beta^{u+1}}{\alpha - \beta}=0.
  \end{equation*}
    Set $b(u+1)=a(p^u)$, then it follows that for all $n\ge 1$ we have
  \[
    b(n(u+1))=a(p^{n(u+1)-1})=\frac{\alpha^{n(u+1)} - \beta^{n(u+1)}}{\alpha - \beta}=0.
  \]    
   Therefore,
  \begin{align}
    \sum_{n\le \tau}\el{\xi a(p^n) } & = \sum_{n=0}^{\tau-1}\el{\xi b(n+1)} 
     = \left(\sum_{n=0}^{ \floor*{\tau/(u+1)}} \sum_{e=0}^{u} \el{\xi b({n(u+1) +e})}\right) + O(u)\nonumber \\
    & = \floor*{\frac{\tau}{u+1}} + \left(\sum_{e=1}^{u} \sum_{n=0}^{ \floor*{\tau/(u+1)}}  \el{\xi b({n(u+1) +e})}\right) + O(u).
    \label{sum_b(n(u+1)+e)} 
  \end{align}    
    First of all observe that $u$ is odd.  As otherwise, if $u$ is even then we would get 
  \[
    \alpha^{u+1}+\beta^{u+1}=2\alpha^{u+1}=\pm 2p^{\frac{(u+1)(k-1)}{2}},
  \]
    which is absurd as $\alpha^{u+1}+\beta^{u+1}$ is an integer but $\frac{(u+1)(k-1)}{2}$ is not. Now, for any $0<e<u+1$ we have
  \[
    b((u+1)n+e)=\alpha^{(u+1)n}\frac{(\alpha^e-\beta^e)}{\alpha-\beta}=\left(\pm p^{\frac{(u+1)(k-1)}{2}}
    \right)^{n}a(p^{e-1}),
  \]
    where the sign on the right hand side above depends on the sign of $\alpha^{u+1}.$ 
    Without loss of generality, we are assuming that this sign is negative. Moreover, it is easy to see that our next argument 
    applies to the positive sign case as well. Since $u$ is fixed, so are all the $e$'s up to $u-1.$ 
    In particular, we may consider large $\ell$'s for which all of the $a(p^{e}) \not \equiv 0 \pmod{\ell}$
    for any $1 \le e \le u-1.$ Then, we have
    \label{10}
  \begin{equation*}\label{eq:ppower}
    \sum_{n=0}^{ \tau/(u+1)}  \el{\xi b({n(u+1) +e})} = 
    \sum_{n=0}^{ \tau/(u+1)} \el{\xi \left(-p^{\frac{(u+1)(k-1)}{2}}\right)^n a(p^{e-1})}.
  \end{equation*}
    Due to Lemma~\ref{Igor1}, we may assume that $t_u=\ord{(-p^{(k-1)(u+1)/2})} >\ell^{\varepsilon}$ holds for $\pi(y)+O(y^{2\varepsilon})$ many primes $\ell \leq y$. Now, by \cite[Theorem 6]{BGK2006} it follows that
  \begin{equation}\label{eq:BGK}
    \left|\sum_{n\le t_u} \el{\xi \left(-p^{\frac{(u+1)(k-1)}{2}}\right)^n a(p^{e-1})} \right| \le t_u \ell^{-\delta}, \quad 
    \textrm{for some }  \delta=\delta(\varepsilon/2)>0.
  \end{equation}
   
   Writing $[\tau/(u+1)]=q t_u + r,$ with 
    $0\le r < t_u$ it follows that
  \begin{align*}
    \sum_{n\le \tau/(u+1)}\el{\xi \alpha^{{(u+1)n}}a(p^{e-1})}  &= 
    q \sum_{n\le t_u}\el{\xi \alpha^{{(u+1)n}}a(p^{e-1})} + \\
    & \;\, + \sum_{n\le r}\el{\xi \alpha^{{(u+1)n}}a(p^{e-1})}.
  \end{align*}
    The estimate
  $
    \left|\sum_{n\le t_u}\el{\xi \alpha^{{(u+1)n}}a(p^{e-1})} \right| \le t_u \ell^{-\delta}
  $    
    follows from \eqref{eq:BGK}. If $r \le \ell^{\varepsilon/2},$ then we get trivially
    $\left| \sum_{n\le r} \el{\xi \alpha^{{(u+1)n}}a(p^{e-1})} \right| 
    \le \ell^{\varepsilon/2}.$ If $\ell^{\varepsilon/2} \le r < t_u,$ then from \eqref{eq:BGK} it follows that 
  \[
    \left| \sum_{n\le r} \el{\xi \alpha^{{(u+1)n}}a(p^{e-1})} \right| \le t_u \ell^{-\delta}.
  \]
    Therefore,
  \begin{align*}
    \left| \sum_{n\le r} \el{\xi \alpha^{{(u+1)n}}a(p^{e-1})} \right|  
    \le \max\left\{\ell^{\varepsilon/2}, t_u \ell^{-\delta}\right\}.
  \end{align*}
    Recalling that $t_u \ge \ell^{\varepsilon},$ we can also assume that $t_u \ell^{-\delta}\ge \ell^{\varepsilon/2}$ 
    by taking small enough $\delta.$ Thus,
  \begin{align*}
    \left|\sum_{n\le \tau/(u+1)}\el{\xi \alpha^{{(u+1)n}}a(p^{e-1})} \right|
     \le (q t_u+ t_u) \ell^{-\delta} \ll \frac{\tau}{u+1}\ell^{-\delta}.
  \end{align*}
     Finally, combining the above inequality with  
    \eqref{sum_b(n(u+1)+e)} we obtain
     \begin{align*}
\max_{\xi \in \rF_{\ell}^*}\left| \sum_{n\le \tau}\el{\xi a(p^n) } \right| &=\floor*{\frac{\tau}{u+1}} + O\left(\tau \ell^{-\delta}+u\right)\\
&=\frac{\tau}{u+1}+O\left(\tau \ell^{-\delta}+u\right).
   \end{align*}
    This conclude the proof for all exceptional set of primes $p\in \rP \setminus \cP.$

\subsection{Consequences of Theorem~\ref{thm:main1-1}}
Let us consider an exponential sum of type $S(p,x,\alpha)=\sum_{p^n\leq x} \mathbf{e}(\alpha a(p^n)),$ for $\alpha \in [0,1].$ As one of the consequences of Theorem~\ref{thm:main1-1}, we want to study this exponential sum when $\alpha$ is a rational whose denominator is a prime. In this regard, we have the following result.

\begin{corollary} Let $f$ be an eigenform of weight $k$ and level $N$ with rational coefficient. Then for a given $ 0 < \varepsilon < 1/2,$ there exists a $\delta(\varepsilon)>0$ such that for at least $\gg \frac{(\log x)^{1 - \delta/(2+\delta)}}{\log \log x}$ many primes $\ell,$ we have the following estimates: 

\begin{equation*}
\max_{\xi\in \mathbb{F}_{\ell}^{*}}\left|\sum_{p^n \leq x}\el{\xi a(p^n)}\right|=\left\{\begin{array}{cccccccc}  
 O\left((\log x / \log p)^{1 - \delta/(2+\delta)}\right) & \mr{if} \quad p\notin \mathcal{P} \\  
 & \\
 \frac{1}{u+1}\frac{\log x}{\log p}+O\left((\log x / \log p)^{1 - \delta/(2+\delta)}\right) & \mr{if } \quad p\in \mathcal{P}
\end{array}\,. \right .
\end{equation*}
\end{corollary}

\begin{proof} 
Consider the same $\delta:=\delta(\varepsilon)$ as in Theorem~\ref{thm:main1-1} and any prime  
   \[\ell  \in \left[(\log x / \log p)^{1/2 - \delta/ (4+2\delta)},2 (\log x / \log p)^{1/2 - \delta/ (4+2\delta)}\right].\]
  Following Theorem~\ref{thm:main1-1}, we have 
 \begin{equation}\label{eq:100}
    \max_{\xi\in \mathbb{F}_{\ell}^{*}}\left|\sum_{p^n \leq \tau}\el{\xi a(p^n)}\right| \leq 
    \frac{\tau}{\ell^{\delta}}\end{equation}
   holds, for at least $\gg \frac{(\log x)^{1 - \delta/(2+\delta)}}{\log \log x}$ primes $\ell.$ For these primes, we also have $\tau \leq \ell^2<\frac{\log x}{\log p}.$ In particular, 
    \[
    \max_{\xi\in \mathbb{F}_{\ell}^{*}}\left|\sum_{p^n \leq x}\el{\xi a(p^n)}\right| \leq \frac{\log x}{ \ell^{\delta} \log p}+O\left(\ell^2\right)=O\left((\log x / \log p)^{1 - \delta/(2+\delta)}\right).\]

On the other hand, let $p\in \mathcal{P}$ be a prime, then by Theorem~\ref{thm:main1-1} we have
\begin{equation*}\label{101}
    \max_{\xi\in \mathbb{F}_{\ell}^{*}}\left|\sum_{p^n \leq \tau}\el{\xi a(p^n)}\right|=\frac{\tau}{u+1}+O\left(\frac{\tau}{\ell^{\delta}}+u\right),
\end{equation*}
holds, for some $u$ depending on $p,$ and for at least $\gg \frac{(\log x)^{1 - \delta/(2+\delta)}}{\log \log x}$ primes $\ell.$ Due to Lemma~\ref{Igor1}, we can assume that $\tau>\ell^{\delta}$ holds by choosing small enough $\delta,$ for at least $\gg \frac{(\log x)^{1 - \delta/(2+\delta)}}{\log \log x}$ primes $\ell.$ Arguing similarly as in the previous case, we get the desired main term, and the error term that we get
\[O\left(\frac{\log x}{\ell^{\delta}\log p}+\frac{u\log x}{\tau \log p}\right)=O\left(\frac{\log x}{\ell^{\delta}\log p}\right)=O\left((\log x / \log p)^{1 - \delta/(2+\delta)}\right),\]
where the last equality holds because $\tau>\ell^{\delta}.$
\end{proof}
  
\begin{corollary}
    Let $f$ be an eigenform of weight $k$ and level $N$ with rational coefficients. 
    For $\pi(y)+O(y^{2\varepsilon})$ many primes $\ell\leq y$ we have the following property. Given $0 < \varepsilon < 1/2$
    and $p_1,\cdots,p_{\nu}$
    be any set of distinct primes such that  $a(p_i^u) \neq 0$
    for all $u\ge1$ and  $1 \le i \le \nu,$ there exists a $\delta=\delta(\varepsilon)>0$ such that
  \[
    \max_{\xi \in \rF_{\ell}^*} 
    \left|\sum_{n_1 \le \tau_1} \cdots \sum_{n_{\nu} \le \tau_{\nu}}  \el{\xi a(p_1^{n_1} \cdots p_{\nu}^{n_{\nu}})}
    \right| \le \tau_1 \cdots \tau_{\nu} \ell^{-\delta}.
  \]
\end{corollary}

\begin{proof} 
    Set 
  \[
    S_{\nu}(\xi)=\left| \sum_{n_1 \le \tau_1} \cdots \sum_{n_{\nu} \le \tau_{\nu}}  \el{\xi a(p_1^{n_1} \cdots p_{\nu}^{n_{\nu}})} \right|.
  \]    
    We proceed by induction. Case $\nu=1$ is done by Theorem~\ref{thm:main1-1}. Now, by multiplicativity
    it follows that
    {\small
  \begin{multline*}
      |S_{\nu}(\xi)|
      \le  \sum_{n_1 \le \tau_1} \left| \sum_{n_2 \le \tau_2} \cdots \sum_{n_{\nu} \le \tau_{\nu}}  
      \el{ \xi a (p_{1}^{n_{1}}) a(p_2^{n_2} \cdots p_{\nu}^{n_{\nu}})} \right| \\
      \le \tau_2 \cdots \tau_{\nu} \hspace{-4mm} \sum_{\substack{{n_1 \le \tau_1} \\ a(p_1^{n_1}) \equiv 0 \hspace{-0.2cm} \pmod{\ell}}} \hspace{-4mm} 1 + 
       \hspace{-0.1cm} \sum_{\substack{{n_1 \le \tau_1} \\ 
      a(p_1^{n_1})\not \equiv 0 \hspace{-0.2cm} \pmod{\ell} }} \left| \sum_{n_2 \le \tau_2} \cdots \sum_{n_{\nu} \le \tau_{\nu}}  
      \el{ \xi a (p_{1}^{n_{1}}) a(p_2^{n_2} \cdots p_{\nu}^{n_{\nu}})} \right|
  \end{multline*}  }
  
    By induction hypothesis, the second term on the right hand side of the above equation is bounded by $\tau_1\tau_2 \cdots \tau_{\nu}\ell^{-\delta},$
    for some $\delta >0$ depending on $\varepsilon.$ On the other hand, note that 
    $\sum_{\substack{{n_1 \le \tau_1} \\ a(p_1^{n_1}) \equiv  0 \hspace{-1mm} \pmod{\ell}}}  \hspace{-0.2cm}1 $
    counts the number of solutions of the congruence
  \[
    a(p^n_1) \equiv 0 \hspace{-2mm}\pmod{\ell}, \qquad n_1 \le \tau_1.
  \]    
   Writing it as exponential sum we get
  \begin{align*}
    \sum_{\substack{{n_1 \le \tau_1} \\ a(p_1^{n_1}) \equiv 0 \hspace{-2mm} \pmod{\ell} }}1  & = \frac{1}{\ell} \sum_{x =0}^{\ell-1}\sum_{n_1 \le \tau_1}  \el{x (a(p_1^{n_1}) ) }  \\
      & = \frac{\tau_1}{\ell} + O\left(\max_{x \in \rF_{\ell}^*} \left|\sum_{n_1 \le \tau_1}  \el{x (a(p_1^{n_1}) ) } \right|  \right).
  \end{align*}
    
    We can bound the error term by Theorem~\ref{thm:main1-1} and without loss of generality assuming $\delta<1,$ we get the sum above is simply $ \frac{\tau_1}{\ell} + O(\tau_1 \ell^{-\delta}).$ This is further bounded by $ 2\tau_1\ell^{-\delta},$ because the explicit constant in Theorem~\ref{thm:main1-1} is exactly $1.$ Therefore,
  \begin{align*}      
      |S_{\nu}(\xi)|\le \tau_2\cdots \tau_v\left(2\tau_1\ell^{-\delta}\right)
       +~\tau_1\tau_2 \cdots \tau_{\nu}\ell^{-\delta},
  \end{align*}  
for some $\delta=\delta(\varepsilon)>0.$ This shows that the inequality
\[
    \max_{\xi \in \rF_{\ell}^*}
    \left|\sum_{n_1 \le \tau_1} \cdots \sum_{n_{\nu} \le \tau_{\nu}}  \el{\xi a(p_1^{n_1} \cdots p_{\nu}^{n_{\nu}})}
    \right| \le 3\tau_1 \cdots \tau_{\nu} \ell^{-\delta}
  \]
  holds for almost all prime $\ell$ and this completes the proof because we can remove the extra factor $3$ by taking large enough $\ell$'s.
\end{proof}

\section{Exponential sums for modular forms : beyond eigenforms}\label{sec:mf}
We shall now prove Theorem~\ref{thm:main1-2}. Write 
  \[
    a_{f}(p^n)=\sum_{i=1}^{r} a_i a_{f_i}(p^n),
  \]
    for some $a_i\in\mathbb{Q},$ where $f_i$'s are newforms with rational coefficients. Let $\omega^{(i,p)}$ to be the characteristic  
    polynomial of $a_{f_i}(p^n)$ and $D_i(p)$ to be its discriminant.
 
    Consider 
  \[
    \mathcal{S}_1=\left\{\ell~\text{prime} \mid \left(\frac{D_i(p)}{\ell}\right)=1, \forall 1\leq i\leq r.\right\}.
  \]
    It is clear that $\mathcal{S}_1$ has positive density. One can verify this by considering primes congruent to $1$ modulo $8\prod_{i=1}^{r}D_i(p).$ This works well because, we then have $\left(\frac{-1}{\ell}\right)=1, \left(\frac{2}{\ell}\right)=1$ and $\left(\frac{\ell}{\mathrm{odd}(D_i(p))}\right)=1, \forall 1\leq i\leq r,$ where $\mathrm{odd}(.)$ denotes odd part of the corresponding number. These conditions altogether implies $\ell \in \mathcal{S}_1.$ Let $\alpha^{(i,p)}$ and $\beta^{(i,p)}$ be the roots of $\omega^{(i,p)}.$ So for any $\ell \in \mathcal{S}_1,$ 
    we can write
  \[
    \omega^{(i,p)}(x) \hspace{-0.1cm}\pmod{\ell}=\prod_{1\leq i\leq
    r}\left(x-\alpha^{(i,p)}_{\ell}\right)\left(x-\beta_{\ell}^{(i,p)}\right),
  \]
    where all of $\alpha^{(i,p)}_{\ell},\beta^{(j,p)}_{\ell}$'s are all in $\mathbb{F}_{\ell}.$ Now, we consider the set of primes
     \begin{align*}
  \mathcal{S}_2 =  & \left\{p \mid \alpha^{(i,p)}(\beta^{(j,p)})^{-1}~\text{are not root of unity,}~\forall~i,j\right\}\\
  & \cup \left\{ p \mid \alpha^{(i,p)}(\alpha^{(j,p)})^{-1}~\text{are not root of unity,}~\forall i\neq j\right\}.
  \end{align*}
  
\begin{lemma}\label{lem:mod form} For any prime $p\in \mathcal{S}_2,$ the following inequalities are true for $\pi(y)+O(y^{2\varepsilon})$ many primes $\ell \leq y.$
  \begin{enumerate}
      \item $\ord(\alpha^{(i,p)}_{\ell}(\beta^{(j,p)}_{\ell})^{-1})>\ell^{\varepsilon},$ 
    $\ord(\alpha^{(i,p)}_{\ell})>\ell^{\varepsilon}$ and $\ord(\beta^{(j,p)}_{\ell})>
    \ell^{\varepsilon}, $ for all $1 \leq i,j \leq r$, and
  \item $\ord(\alpha^{(i,p)}_{\ell}(\alpha^{(j,p)}_{\ell})^{-1})>\ell^{\varepsilon},$ for all $1\leq i \neq j \leq r,$
   \end{enumerate}
    
\end{lemma}

\begin{proof} 

It is enough to prove the result only for $i,j \in \{1,2\}.$ Consider the Galois extension $K=\mathbb{Q}\left(\alpha^{(1,p)},\alpha^{(2,p)}\right).$ Let $\mathfrak{L}$ be a prime ideal lying over $\ell$ in $\mathcal{O}_K.$ It is clear that
\begin{equation}\label{eq:two sides}
\{\alpha^{(1,p)}_{\ell}, \alpha^{(2,p)}_{\ell}, \beta^{(1,p)}_{\ell}, \beta^{(2,p)}_{\ell}\}=\{\alpha^{(1,p)}, \alpha^{(2,p)}, \beta^{(1,p)}, \beta^{(2,p)}\} \pmod {\mathfrak{L}},
\end{equation}

because both of these sets serve as a set of roots of the equation $\omega(x)\hspace{-0.1cm}\pmod{\ell}$ and $\omega(x)\hspace{-0.1cm}\pmod{\mathfrak{L}}$ respectively. Note that $\omega(x)\hspace{-0.1cm}\pmod{\mathfrak{L}}$ coincides with $\omega(x)\hspace{-0.1cm}\pmod{\ell}.$ It follows from~\eqr{two sides} that the right hand side does not depend on the choice of prime $\mathfrak{L}$ 
lying over $\ell,$ so there is no problem in working with a fixed $\mathfrak{L}$ lying over $\ell.$ It is now clear that,
\[\left\{\alpha^{(i,p)}_{\ell}(\beta^{(j,p)}_{\ell})^{-1}\right\}_{1\leq i,j\leq 2}=\left\{\alpha^{(i,p)}(\beta^{(j,p)})^{-1}\right\}_{1\leq i,j\leq 2} \pmod {\mathfrak{L}}.\]
Consider $R(T)=\mathrm{Res}\left(\omega_1(x), g_T(x)\right),$ where $\omega_1(x)=\left(x-\alpha^{(1,p)}\right)\left(x-\beta^{(1,p)}\right)$ and \[g_T(x)=\prod_{t\leq T}\left(x^t-\alpha^{(2,p)t}\right)\left(x^t-\beta^{(2,p)t}\right).\]
  It is clear that $R(T)\neq 0$ for any $T\in \mathbb{N}$ as $p\in \mathcal{S}_2$ by assumption. Now consider the set of primes,
\begin{equation}\label{eq:T}
  \left\{\ell \mid \mathrm{ord}\left(\alpha^{(i,p)}_{\ell}(\beta^{(j,p)}_{\ell})^{-1}\right), \mathrm{ord}\left(\alpha^{(i,p)}_{\ell}(\alpha^{(j,p)}_{\ell})^{-1}\right) \leq 
  T~\textrm{for some}~i \neq j\in \{1,2\}\right\}.
  \end{equation}
For any prime $\ell$ in the set above, and for any prime $\mathfrak{L}$ in $\mathcal{O}_K$ lying over $\ell,$ $\omega_1(x) \hspace{-0.1cm}\pmod{\mathfrak{L}}$ and $g_T(x)\hspace{-0.1cm} \pmod {\mathfrak{L}}$ have a common root, 
Therefore, $R(T)\pmod {\mathfrak{L}}=0.$ Since both $\omega_1$ and $g_T(x)$ are in $\mathbb{Z}[x],$ it is clear that $R(T) \in \mathbb{Z},$ and so $R(T) \hspace{-0.1cm}\pmod \ell =0$ as well. Now one can estimate the number of prime divisors of $R(T)$ similarly as in Lemma~\ref{Igor1}. This shows that
\[\mathrm{ord}\left(\alpha^{(i,p)}_{\ell}(\beta^{(j,p)}_{\ell})^{-1}\right) > \ell^{\varepsilon},~\mathrm{and}~\mathrm{ord}\left(\alpha^{(i,p)}_{\ell}(\alpha^{(j,p)}_{\ell})^{-1}\right) > \ell^{\varepsilon}\]
holds for all $i\neq j \in \{1,2\},$ and $\pi(y)+O(y^{2\varepsilon})$ many primes $\ell \leq y.$ Rest of the cases can be dealt using Lemma~\ref{Igor1}.
\end{proof} 

\subsection{GST: Beyond Sato-Tate}\label{se:gst}
We shall now give a short overview of Sato-Tate distribution. When $f$ is a newform without $CM$, then Sato-Tate conjecture says that the normalized coefficients $\frac{a(p)}{2p^{\frac{k-1}{2}}}$ are equidistributed in $[-1,1]$ with respect to the measure
  \[\mu_{\mathrm{non}-CM}=\frac{2}{\pi}\int \sin^2(\theta) \ud \theta. \]
On the other hand, if $f$ is with $CM$, then the corresponding Sato-Tate distribution is
  \[\mu_{CM}=\frac{1}{2\pi}\int \frac{\ud x}{\sqrt{1-x^2}}=\frac{1}{2\pi}\int 1 \ud \theta,\]
on $[0,\pi]-\{\frac{\pi}{2}\}$. Moreover at $\theta_p=\frac{\pi}{2},$ $a(p)$ becomes zero and it is known that the set of such primes $p$ have density exactly $\frac{1}{2}.$ Now consider the $L$-function defined by
\[L(s,\mathrm{Sym}^m f)=\prod_{p\nmid N}\prod_{i=0}^{m}\left(1-\alpha_p^i\beta_p^{m-i}p^{-s}\right)^{-1},\]
where $\alpha_p,\beta_p$ are normalized roots of (\ref{eq:associate}). In other words, if $\widetilde{\alpha}_p,\widetilde{\beta}_p$ be the roots of (\ref{eq:associate}), then we define $\alpha_p=\frac{\widetilde{\alpha}_p}{p^{\frac{k-1}{2}}}, \beta_p=\frac{\widetilde{\beta}_p}{p^{\frac{k-1}{2}}}.$ Serre in~\cite{Serre1998}~showed that if for all integer $m \geq 0, L(s, \mathrm{Sym}^m(f))$ extends analytically to $\mathrm{Re}(s) \geq 1$ and does
not vanish there, then the Sato–Tate conjecture holds true for $f.$ Note that Barnet-Lamb et al. have proved the conjecture in~\cite{BGHT} working with this $L$-function. 

However, in the next lemma we will have more than one newform to play with, and it will be helpful to have their distributions independent. \emph{\textbf{We are stating this independency as Generalized Sato-Tate (GST) hypothesis.}} It can be shown that if we have newforms $f_1, f_2, \cdots, f_r,$ then their Sato-Tate distributions are independent to each other provided that the Rankin-Selberg $L$-function (see~\cite{Shahidi} for a definition)
\[L(s,\mathrm{Sym}^{m_1}f_1\otimes\cdots \mathrm{Sym}^{m_r}f_r )\]
extends to $\mathrm{Re}(s)\geq 1$ and does not vanish for all non negative integers $m_1,\cdots, m_r.$ If one of the $f_i$ is with $CM,$ then we know due to Ribet~\cite{Ribet1977} that $L$-function of $f_i$ comes from $L$-function associated to a Hecke character. Now suppose that at most one of the $f_i$ is without $CM$, say $f_1.$ Then without loss of generality we can write
 \[L\left(s,\mathrm{Sym}^{m_1}f_1\otimes\cdots \mathrm{Sym}^{m_r}f_r\right)=L\left(s,\mathrm{Sym}^{m_1}f_1\otimes\mathrm{Sym}^{m_2}\psi_2\otimes \cdots \otimes \mathrm{Sym}^{m_r}\psi_r\right),\]
 where $\psi_i$'s are the corresponding Hecke characters. It follows from~\cite[Theorem B.3]{BGHT} that for any odd $m$, 
 there exists a Galois extension $K$ over $\mathbb{Q}$ such that the base change $\mathrm{Sym}^{m}f_1|_{K}$ is automorphic. Following the arguments given on the page 643 of~\cite{Murty-Murty} one can write,
 \begin{align*} & L\left(s,\mathrm{Sym}^{m_1}f_1\otimes\mathrm{Sym}^{m_2}\psi_2\otimes \cdots \otimes \mathrm{Sym}^{m_r}\psi_r\right)=\\
 & \quad = \prod_{i} \left(L(s,(\mathrm{Sym}^{m_i}f_1)|_{K^{H_i}}\otimes \chi_i\otimes\mathrm{Sym}^{m_2}\psi_2\otimes \cdots \otimes \mathrm{Sym}^{m_r}\psi_r\right)^{a_i},
 \end{align*}
 where $H_i$'s are nilpotent subgroups of $\mathrm{Gal}\left(K/\mathbb{Q}\right)$ and $a_i$'s are integers.
 In particular, we now have a meromorphic continuation to $\mathrm{Re}(s)\geq 1.$ It is known that any automorohic $L$-function is non vanishing on $\mathrm{Re}(s)=1,$ 
 in particular we now have the desired analytic continuation to $\mathrm{Re}(s)\geq 1$ for any odd $m_1.$ Now if $m_1$ is even, we argue inductively as in~\cite{Murty-Murty}. 
 The point is, similar to~\cite[pages~$643-644$]{Murty-Murty}, we need to study non vanishing of a Rankin-Selberg $L$-function on $\mathrm{Re}(s)=1,$ which can be done by using Shahidi's result on Rankin-Selberg $L$-function. See $(e)$ at page $418$ of~\cite{Shahidi}.
 
In particular, when there is at most one component without $CM,$ then their corresponding Sato-Tate distributions are independent to each other. In other words, the GST hypothesis is true in this particular case.

\begin{lemma}\label{int:density1} Suppose that there are $r_1$ many components without $CM$ and $r_2$ many components with $CM$ in $f.$ Then under the GST hypothesis, density of $\mathcal{S}_2$ is $2^{-r_2}.$
 \end{lemma}
 
 \begin{proof} We start by writing
 \[\alpha^{(j,p)}=p^{\frac{k-1}{2}}e^{i\theta_{j,p}},\beta^{(j,p)}=p^{\frac{k-1}{2}}e^{-i\theta_{j,p}}, \forall 1\leq j\leq r.\]
 So, the problem reduced to study the set of primes
 \begin{equation}\label{dens}
 \left\{p \mid \theta_{i,p}\pm \theta_{j,p}\in \mathbb{Q}\times \pi,~\text{for some}~1\leq i,j\leq r\right\}.
 \end{equation}

 It follows from the discussion above that the density of this set is bounded by
\begin{equation}\label{int S}
  \left(\frac{2}{\pi}\right)^{r_1}\left(\frac{1}{2\pi}\right)^{r_2}\idotsint \limits_{S}\sin^2(\theta_1)\sin^2(\theta_2)\cdots\sin^2(\theta_{r_1})\ud\theta_1 
    \ud\theta_2\cdots \ud\theta_r,
\end{equation}
where $S=\left\{(\theta_1,\theta_2,\cdots,\theta_r) \in [0,\pi]^r \mid \theta_{i}\pm \theta_{j}\in \mathbb{Q}\times \pi~\text{for some}~1\leq i,j\leq r\right\}.$ Just for the sake of simplicity and to have a feel of what is going on, let us first do the case when there is only one component.
\subsection*{Case 1, $r=1$~:} suppose that the given component is without $CM.$ If $\alpha_p^{(1,p)}\beta_p^{-(1,p)}$ 
    is a root of unity then this implies that $\theta_{1,p} \in \pi \times \mathbb{Q}.$ 
    By Sato-Tate, density of such primes is bounded by
  \[
    \left(\frac{2}{\pi}\right)\int \limits_{\theta\in \pi \times \mathbb{Q}}\sin^2(\theta) \ud \theta.
  \] 
  Since the integral above runs over a set of measure zero, the integral is zero, and this particular case density of $\mathcal{S}_2$ is indeed $1.$ Now suppose the given component is with $CM$. In this case, the density of $\mathcal{S}_2$ is 
  \[
    \left(\frac{1}{2\pi}\right)\int \limits_{\theta \in [0,\pi]\setminus \pi \times \mathbb{Q}}\sin^2(\theta) \ud \theta=\frac{1}{2}.\]
 \subsection*{Case 2,~$r\geq 2$~:} for this general case, it is enough to show that the integral over $S$ at $(\ref{int S})$ is zero. This is because, due to GST, we are now working with the measure
\begin{equation}\label{int Ss}
  \left(\frac{2}{\pi}\right)^{r_1}\left(\frac{1}{2\pi}\right)^{r_2}\idotsint \sin^2(\theta_1)\sin^2(\theta_2)\cdots\sin^2(\theta_{r_1})\ud\theta_1 
    \ud\theta_2\cdots \ud\theta_r,
\end{equation}
and with respect to this measure, $[0,\pi]^r$ has measure $\left(\frac{1}{2}\right)^{r_2}.$ We can write $S=\bigcup_{1\leq i,j\leq r} S_{i,j}$ where the set $S_{i,j}$ is defined to be the tuples for which $\theta_i\pm \theta_j\in \mathbb{Q}\times \pi.$ It is now enough to show that each of $S_{i,j}$'s have zero measure. Note that, the integral over $S_{i,j}$ is crudely bounded by $\iint \limits_{S_{i,j}}1\ud \theta_i \ud \theta_j.$ It is evident that
  \begin{align*}
    \iint \limits_{S_{i,j}}1\ud\theta_i \ud\theta_j & =
         \iint \limits_{\theta_i+\theta_j \in \mathbb{Q}
         \times \pi}1\ud \theta_i \ud\theta_j \; + \iint \limits_{\theta_i-\theta_j \in \mathbb{Q}\times \pi}1\ud\theta_i \ud\theta_j,
  \end{align*}
   as $\mathbb{Q}\times \mathbb{Q}$ has zero measure. We now note that,
  \begin{equation}\label{int 2 parts}
    \iint \limits_{\theta_i-\theta_j \in (a,b)}1\ud\theta_i \ud\theta_j
    \leq \int \limits_{0}^{\pi}\int \limits_{a}^{b}1 \ud t \ud\theta \ll |b-a|,
 \end{equation}
for any $b>a.$ In particular, for any $\varepsilon>0,$
  \[
      \iint \limits_{\theta_i - \theta_j \in \mathbb{Q}\times
      \pi}1\ud\theta_i\ud\theta_j
      \ll \sum_{k=1}^{\infty} \frac{\varepsilon}{2^k}=\varepsilon
  .\]
    The last implication above follows from the standard argument to show a countable set 
    always has zero measure. In particular, the second integral of (\ref{int 2 parts}) is zero. On the other hand, just by replacing $\theta_j$ with $\pi-\theta_j,$ we get 
    \[\iint \limits_{\theta_i+\theta_j \in \mathbb{Q}\times \pi}\ud\theta_i \ud\theta_j=-\iint \limits_{\theta_i - \theta_j \in \mathbb{Q}\times
      \pi}1\ud\theta_i\ud\theta_j.\]
    This just shows that the integral over $S_{i,j}$ at (\ref{int 2 parts}) is zero, which completes the proof.

 \end{proof}

\subsection{Proof of Theorem~\ref{thm:main1-2}} Let $p\in \mathcal{S}_2$ be a prime, then we can write
\[\sum_{i=1}^{r} a_ia_{f_i}(p^n)\hspace{-0.2cm} \pmod{\ell}=\sum_{i=1}^r a_i^{(\ell)}\left(c^{(i,\ell})\alpha^{n(i,\ell)}+d^{(i,\ell)}\beta^{n(i,\ell)}\right),\]
where $a_i^{(\ell)}, c^{(i,\ell)}$ and $d^{(i,\ell)}$ are all in $\mathbb{F}_{\ell}.$ On the other hand all $\alpha^{(i,\ell)}$ and $\beta^{(i,\ell)}$'s are in $\mathbb{F}_{\ell},$ as $\ell \in \mathcal{S}_1.$ The proof now follows by~\cite[Corollary]{Bourgain2005a} joint with Lemma~\ref{lem:mod form} and Lemma~\ref{int:density1}. 
\qed

\section{Exponential sums for  modular forms : the inverse case }\label{se:inverse case}
One may now ask that for a given prime $\ell$ and small enough $\varepsilon,$  how many primes $p$ are there for which an estimate like~\eqr{main1} holds. Our attempt to answer this question is summarized in the form of Theorem~\ref{thm:main2-1} and  Theorem~\ref{thm:main2-2}. Let us begin with the proof of Theorem~\ref{thm:main2-1}.

\subsection{Proof of Theorem~\ref{thm:main2-1}}
For any prime $p$, let us denote the roots of $x^2-a(p)x+p^{k-1}\hspace{-0.1cm} \pmod{\ell}$ by $\alpha_{p}^{(\ell)}, \beta_{p}^{(\ell)}.$ Recall that from Deligne-Serre correspondence, we have the associated Galois representation
\[\rho_{f}^{(\ell)}:\text{Gal}\left(\overline{\mathbb{Q}}/\mathbb{Q}\right) \longrightarrow \mathrm{GL}_2\left(\mathbb{Z}_{\ell}\right),\]
such that $a(p)=\text{tr}\left(\rho_{f}^{(\ell)}(\text{Frob}_p)\right)$ for any prime $p \nmid N\ell.$ It is clear that the characteristic polynomial of $\rho_f^{(\ell)}(\text{Frob}_p)\pmod{\ell}$ is same as $x^2-a(p)x+p^{k-1}\hspace{-0.1cm}\pmod{\ell}$. Following Ribet~\cite[Theorem 3.1]{Ribet85}, it is known that the image of this representation is $\left\{A \in \text{GL}_2\left(\mathbb{Z}_{\ell}\right) \mid \det(A) \in (\mathbb{Z}_{\ell}^{*})^{k-1}\right\},$ except possibly for finitely many primes $\ell.$ In particular, the condition $(k-1,\ell-1)=1$ implies that the induced Galois representation 
\[\rho_{f,\ell}:\text{Gal}\left(\overline{\mathbb{Q}}/\mathbb{Q}\right) \longrightarrow \mathrm{GL}_2\left(\mathbb{F}_{\ell}\right),\] 
is surjective for any large prime $\ell,$ and the eigenvalues of the matrix $\rho_{f,\ell}(\text{Frob}_p) \in\text{GL}_2\left(\mathbb{F}_{\ell}\right)$ 
are $\alpha_{p}^{(\ell)}$ and $\beta_{p}^{(\ell)}.$ From the proof of Theorem~\ref{thm:main1-1}, we know that an estimate of type~\eqr{main1} holds provided that, 
$$\ord(\alpha_{p}^{(\ell)}) > \ell^{\varepsilon}, \ord(\beta_{p}^{(\ell)}) > \ell^{\varepsilon},~\text{and}~\ord(\alpha_{p}^{(\ell)}(\beta_{p}^{(\ell)})^{-1}) > \ell^{\varepsilon}.$$

Let us define, 
  \[
    C=\left\{ A\in \mathrm{GL}_2(\mathbb{Z}/\ell\mathbb{Z}) \mid \ord(\lambda_{1,A}), \ord(\lambda_{2,A}), \ord(\lambda_{1,A} \lambda_{2,A}^{-1})>\ell^{\varepsilon}\right\},
  \]
  where $\lambda_{1,A},\lambda_{2,A}$ are the eigenvalues of $A$ in $\mathbb{F}_{\ell^2}.$ Now the problem is about computing the density of primes $p$ for which the corresponding
  $\rho_{f,\ell}\left(\text{Frob}_p\right)$ is in $C.$ Note that $C$ is a subset of $\text{GL}_2(\mathbb{F}_{\ell})$
  stable under conjugation. Hence, by Chebotarev's density theorem, the required density is at least
  $\frac{|C|}{|\text{GL}_2\left(\mathbb{F}_{\ell}\right)|}.$ For each $a \neq b \in \mathbb{F}_{\ell}^{*},$ let $C_{a,b}$ 
   be the conjugacy class of $\tmt{a}{0}{0}{b}.$ It is known that $|C_{a,b}|=(\ell+1)\ell.$ For any element $\lambda$ in $\mathbb{F}_{\ell^2}\setminus \mathbb{F}_{\ell},$ we denote $c_{\lambda}$ to be the conjugacy class of matrices in $\text{GL}_2\left(\mathbb{F}_{\ell}\right)$ having eigenvalue $\lambda.$ It is known that $|C_{\lambda}|=\ell(\ell-1).$ Now, we consider the following sets:
\begin{align*}
S_1 &=\left\{ a,b \in \mathbb{F}_{\ell}^{*} \mid \ord(a)>\ell^{\varepsilon}, \ord(b)>\ell^{\varepsilon}, \ord(ab^{-1})>\ell^{\varepsilon}\right\},\\ 
S_2&=\left\{\lambda \in \mathbb{F}_{\ell^2}^{*} \setminus \mathbb{F}_{\ell}^{*} \mid \ord(\lambda)=\ord(\lambda^{\ell})>\ell^{\varepsilon}, \ord(\lambda^{\ell-1}) > \ell^{\varepsilon}\right\}, 
\end{align*}
and realize that $|C|=\frac{1}{2}((\ell+1)\ell|S_1|+\ell(\ell-1)|S_2|).$ This boils down to the problem of estimating $S_1$ and $S_2.$ Let us first estimate $S_1.$ For any divisor $d$ of $\ell-1,$ the set of all elements of $\mathbb{F}_{\ell}^*$ having order exactly $d$ is of the form $\sigma^{\frac{\ell-1}{d}i}$ with $(i,d)=1.$ In particular, the number of elements of $\mathbb{F}_{\ell}^*$ with order greater than $\ell^{\varepsilon}$ is given by
\[\sum_{\substack{d\mid \ell-1\\ d>\ell^\varepsilon}} \phi(d)=\ell+O\left(\sum_{\substack{d\mid \ell-1\\ d<\ell^\varepsilon}} \phi(d)\right)=\ell+O\left(\ell^{\varepsilon}d(\ell-1)\right)=\ell+O_{\varepsilon}\left(\ell^{2\varepsilon}\right),\]
where $d(\cdot)$ is the divisor function, and here we are using the well known upper bound on divisor function (see~\cite{Murty-Analytic}) for any large enough prime $\ell.$ Now note that $\text{ord}\left(ab^{-1}\right)<\ell^{\varepsilon}$ implies that $ab^{-1}$ belongs to a set with only $\sum_{k\mid \ell-1,k<\ell^{\varepsilon}} \phi(k)$ many elements. By the argument above, this set has only  $O_{\varepsilon}\left(\ell^{2\varepsilon}\right)$ many elements. This observation implies that
\[|\left\{ a,b \in \mathbb{F}_{\ell}^{*} \mid \ord(a),\ord(b),or \ord(ab^{-1})<\ell^{\varepsilon}\right\}|=O_{\varepsilon}(\ell^{2\varepsilon+1}).\]
In particular, we then have $|S_1|=\ell^2+O_{\varepsilon}(\ell^{2\varepsilon+1}).$

Let us now estimate $|S_2|.$ Take $\tau$ to be a generator of $\mathbb{F}_{\ell}^{*},$ then any $\lambda \in S_2$ is of the form $\tau^{\frac{\ell^2-1}{d}i},$ with $(i,d)=1.$ Moreover, we also have an order restriction on $\lambda^{\ell-1},$ which implies that $\frac{d}{(d,\ell-1)}>\ell^{\varepsilon}.$ Hence,$$ |S_2|=\sum_{\substack{d\mid \ell^2-1 \\ \frac{d}{(d,\ell-1)}>\ell^{\varepsilon}}} \phi(d)=\ell^2+O\bigg(\sum_{\substack{d\mid \ell^2-1 \\ \frac{d}{(d,\ell-1)}<\ell^{\varepsilon}}} \phi(d)\bigg).$$
Note that, the condition $\frac{d}{(d,\ell-1)}<\ell^{\varepsilon}$ implies that $d<\ell^{\varepsilon+1}.$ Therefore, 
\[\sum_{\substack{d\mid \ell^2-1 \\ \frac{d}{(d,\ell-1)}<\ell^{\varepsilon}}} \phi(d) \leq \ell^{\varepsilon+1} d(\ell^2-1)=O_{\varepsilon}\left(\ell^{1+3\varepsilon}\right).\]

Therefore, the required density is at least 
\[\frac{1}{2}(\ell-1)\ell\frac{|S_1|}{|\mathrm{GL}_2\left(\mathbb{F}_{\ell}\right)|}+\frac{1}{2}(\ell+1)\ell\frac{|S_2|}{|\mathrm{GL}_2\left(\mathbb{F}_{\ell}\right)|}=1+O_{\varepsilon}\left(\frac{1}{\ell^{1-3\varepsilon}}\right).\]
\qed

\subsection{Proof of Theorem~\ref{thm:main2-2}} 
Let  $\rho_{f,\ell}:\mathrm{Gal}\left(\overline{\mathbb{Q}}/\mathbb{Q}\right) \to \mathrm{GL}_{2r}\left(\mathbb{F}_{\ell}\right)$ be the map
defined by
\[\sigma \mapsto \begin{psmallmatrix}
    \rho_{f_{1,\ell}}(\sigma) & & \\
    & \rho_{f_{2, \ell}}(\sigma) & &\\
    & & \ddots &\\
    & & & \rho_{f_r,\ell}(\sigma)\\ 
  \end{psmallmatrix}.\] 
It is clear that the image of this representation is contained in $\Delta_r(\ell),$ where 
\[\Delta_r(\ell)=\left\{\begin{psmallmatrix}
    g_1 & & &\\
    & g_2 & &\\
    &  & \ddots & \\ 
    & & & g_r\\ 
  \end{psmallmatrix} \mid \det(g_1)=\det(g_2)=\cdots=\det(g_r)\right\}.\]
It is in fact the case that the image is contained in $\Delta_r^{(k-1)}(\ell),$ where $\Delta^{(k-1)}_{r}(\ell)$ denotes the set of matrices in $\Delta_r(\ell)$ in which determinant of each block is a ${(k-1)}^{th}$ power in $\mathbb{F}_{\ell}^{*}.$ Due to \cite[Theorem 3.1]{Ribet85}, we may assume that for any large enough prime $\ell,$ image of each $\rho_{f_{i,\ell}}$ is $\Delta_1^{(k-1)}(\ell),$ which also coincides with the set of matrices in $\text{GL}_2(\mathbb{F}_{\ell})$ whose determinant is a $(k-1,\ell-1)^{th}$ power in $\mathbb{F}_{\ell}^{*}.$ If image of $\rho_{f,\ell}$ is not exactly $\Delta_r^{(k-1)}(\ell),$ then by \cite[Lemma 5.1]{MW} we get a set of quadratic characters $\{\chi_{i,j,\ell}\}_{1\leq i,j\leq r}$ of $\mathrm{Gal}\left(\overline{\mathbb{Q}}/\mathbb{Q}\right)$ such that
\[\rho_{f_i,\ell}\left(\text{Frob}_p\right)~\mathrm{is~conjugate~ to}~\chi_{i,j,\ell}\left(\text{Frob}_p\right)\rho_{f_j,\ell}\left(\text{Frob}_p\right)~\text{in}~\mathrm{GL}_2(\mathbb{F}_{\ell}),\]
 for all $1 \leq i,j\leq r.$ In particular, $a_i(p)=\pm a_j(p) \pmod \ell,$ for all $1\leq i,j\leq r,$ and any prime $p \nmid N\ell.$ This implies $\alpha^{(i,p)}_{\ell}+\beta^{(i,p)}_{\ell}=\pm(\alpha^{(j,p)}_{\ell}+\beta^{(j,p)}_{\ell}).$ Moreover, we also know $\alpha^{(i,p)}_{\ell}\beta^{(i,p)}_{\ell}=\alpha^{(j,p)}_{\ell}\beta^{(j,p)}_{\ell}=p^{k-1}\hspace{-0.1cm}\pmod \ell.$ In particular, this means
 \begin{equation}\label{eqn:inside}
 \{\alpha^{(i,p)}_{\ell},\beta^{(i,p)}_{\ell}\}=\pm \{\alpha^{(j,p)}_{\ell},\beta^{(j,p)}_{\ell}\}, \forall 1\leq i,j\leq r, \text{and for any prime}~p \nmid N\ell. 
 \end{equation}
 
Due to the assumption regarding GST, for a positive density of primes $p,$ none of the $\{\alpha^{(i,p)}\beta^{-(j,p)}\}_{1\leq i,j\leq 2}$ or $\pm \{\alpha^{(i,p)},\alpha^{-(j,p)}\}_{1\leq i\neq j\leq 2}$ is a root of unity. For those primes $p$, following the arguments in the proof of Lemma~\ref{lem:mod form}, and considering the set at (\ref{eq:T}), each element of the set $\{\alpha^{(i,p)}_{\ell}\beta^{-(j,p)}_{\ell}\}_{1\leq i,j\leq 2}$ have order larger than $4$ except for finitely many primes $\ell.$ We then have a contradiction to (\ref{eqn:inside}), and hence we may assume the image of $\rho_{f,\ell}$ is indeed $\Delta_r^{(k-1)}(\ell)$ for any large enough prime $\ell.$

Hence the required density is at least $\frac{|C_r^{k-1}(\ell)|}{|\Delta_r^{(k-1)}(\ell)|},$ where $C^{k-1}_r(\ell)$ is the conjugacy classes of elements in $\Delta_r^{(k-1)}(\ell)$ whose eigenvalues satisfy the conditions of Theorem~\ref{Thm:Main}. Note that any tuple $\left(a_1,a_2,\cdots, a_{2r}\right)\in (\mathbb{F}_{\ell}^{*})^{2r}$ with $\ord(a_i)>\ell^{\varepsilon},\ord(a_ia_j^{-1})>\ell^{\varepsilon},\forall i\neq j$ and $a_ia_{i+1}=a_ja_{j+1},\forall i,j~\text{odd}$, satisfies that $\prod_{i,~\mathrm{odd}} C_{a_i,a_{i+1}} \subseteq C_r^{k-1}(\ell).$ We call these tuples $\textit{nice}$ and we want to count them. First of all note that,
\[\{\left(a_1,a_2,\cdots, a_{2r}\right)\in (\mathbb{F}_{\ell}^{*})^{2r}\mid a_i a_{i+1}=a_j a_{j+1},\forall i,j~\mathrm{odd}\}=\frac{(\ell-1)^{r+1}}{(\ell-1,k-1)}.\]
On the other hand, for any $(k-1)^{th}$ power $\lambda$ in $\mathbb{F}_{\ell}^{*},$ note that $ab=\lambda$ and $\ord(ab^{-1})<\ell^{\varepsilon}$ implies  $\ord(a^2\lambda^{-1})<\ell^{\varepsilon}.$ From the proof of Theorem~\ref{thm:main2-1}, for a fixed $\lambda$ the number of such $a$ is $O_{\varepsilon}(\ell^{2\varepsilon}).$ Moreover, $\ord(a)<\ell^{\varepsilon}$ or $\ord(b)<\ell^{\varepsilon}$ holds for only $O_{\varepsilon}(\ell^{2\varepsilon})$ many $a$ or $b$’s. In particular, the number of tuples that does not come into our consideration are 
\[\sum_{\lambda,~(k-1)^{th}~\mathrm{power}} O_{\varepsilon}(\ell^{r-1+2\varepsilon})=O_{\varepsilon}\left(\frac{\ell^{r+2\varepsilon}}{(k-1,\ell-1)}\right).\]
In particular, we then have
\begin{align}
|C_r^{k-1}(\ell)|&\geq \sum_{(a_1,a_2,\cdots, a_r)~\mathrm{nice}} \left(\prod_{i~\mathrm{odd}}|C_{a_i,a_{i+1}}|\right)\\
& = \left(\frac{\ell(\ell+1)}{2}\right)^r\left(\frac{(\ell-1)^{r+1}}{(\ell-1,k-1)}+O_{\varepsilon}\left(\frac{\ell^{r+2\varepsilon}}{(k-1,\ell-1)}\right)\right).\nonumber
\end{align}
The extra factor $\left(\frac{\ell(\ell+1)}{2}\right)^r$ is coming because each conjugacy class $C_{a_{i,i+1}}$ has $\ell(\ell+1)$ many elements and taking into consideration that $C_{a_i,a_{i+1}}=C_{a_{i+1},a_i},\forall i~\text{odd},$ the extra factor $\frac{1}{2}$ is coming for each component. The proof is now complete because  $|\Delta_r^{(k-1)}(\ell)|=\left(\frac{|\mathrm{GL}_2(\mathbb{F}_{\ell})|}{\ell-1}\right)^r\frac{\ell-1}{(\ell-1,k-1)}.$
\qed

\section{Impact on Waring-type problems}\label{se:waring}
Given a sequence of integers $\{x_n\}$ one of the classical questions in additive number theory consists of
    deciding whether $\{x_n\}$ is an additive basis in $\rZ$ or over finite fields. More precisely, is there an
    absolute constant $k\ge1$ such that any residue class $\lambda$ modulo $\ell$ can be represented as 
  \[
    x_{n_1} + \cdots + x_{n_k} \equiv \lambda \hspace{-0.2cm}\pmod \ell,
  \] 
    for infinitely many primes $\ell$?  For instance, it is easy to see that the Fibonacci  sequence
  \[
    F_{n+2}=F_{n+1} + F_n, \quad \textrm{with} \quad F_0 =0, \, F_1= 1,
  \]  
    is not an additive basis in $\rZ,$ however the third author proved in~\cite[Theorem 2.2]{Garcia2013},
    that given a parameter $N\to \infty,$ for $\pi(N)(1+o(1))$
    primes $\ell \le N,$ every residue class modulo $\ell$ can be written as
  \[
    F_{n_1}+\cdots + F_{n_{16}} \equiv \lambda\hspace{-0.2cm} \pmod \ell,
  \]  
    provided that $n_1, \ldots, n_{16} \le N^{1/2 + o(1)}.$ The method is based on distribution properties 
    of sparse sequences for almost all primes and particular identities of Lucas sequence. It does not seem easy to 
    extend such ideas for general linear recurrence sequences.

    In the present section we combine Theorem~\ref{Thm:Main} with classical analytical tools to prove that a linear recurrence sequence $\{s_n\}$ is an additive basis over prime fields, under some assumptions. Moreover, we discuss about the advantages of getting nontrivial exponential sums to prove it.
    
\subsection{Waring-type problems with linear recurrence sequences}\label{subse:waring}
    Let $\{s_n\}$ be a  nonzero linear recurrence sequence modulo $\ell$ as in~\eqref{eq:RecurrenceSeq}
    with order $r,$ $(a_0,\ell)=1$ and period $\tau$. Given an integer $k\ge 2,$  for any residue 
    class $\lambda$ modulo $\ell$ we denote by $T_k(\lambda)$ the number of solutions of the congruence
  \[ 
    s_{n_1} + \cdots +s_{n_k} \equiv \lambda \hspace{-0.2cm} \pmod \ell,  \quad \textrm{with} \quad 1\le n_1, \ldots, n_k \le \tau.
  \]
    Then writing $T_k(\lambda)$ in terms of exponential sums we get
    \[T_k(\lambda)=\frac{1}{\ell}\sum_{\xi=0}^{\ell-1} \sum_{n_1 \le \tau} \cdots \sum_{n_k \le \tau} \el{\xi(s_{n_1} + \cdots +s_{n_k} - \lambda)}.\]    
    Taking away the term $\xi=0$ and using triangle inequality it is clear that 
  \begin{align}\label{eq:WaringJ_k}
      T_k(\lambda) & = \frac{\tau^k}{\ell} + \frac{1}{\ell}\sum_{\xi=1}^{\ell-1} \sum_{n_1 \le \tau} \cdots \sum_{n_k \le \tau} \el{\xi(s_{n_1} + 
              \cdots +s_{n_k} - \lambda)} \nonumber \\
          & = \frac{\tau^k}{\ell} + \frac{\theta''}{\ell}\sum_{\xi=1}^{\ell-1}\left| \sum_{n_1 \le \tau} \cdots \sum_{n_k \le \tau} \el{\xi(s_{n_1} + 
              \cdots +s_{n_k})}\right| \nonumber \\
          & = \frac{\tau^k}{\ell} + \frac{\theta'}{\ell}\sum_{\xi=1}^{\ell-1} \left( \left| \sum_{n_1 \le \tau} \el{\xi s_{n_1}} \right| \cdots 
              \left| \sum_{n_k \le \tau} \ep{\xi s_{n_k}}\right| \right) \nonumber \\
          & = \frac{\tau^k}{\ell} + {\theta}\left(\max_{\xi\in \rF_{\ell}^*}\left| \sum_{n \le \tau} \el{\xi s_{n}} \right| \right)^k,
  \end{align}
    where $\theta, \theta'$ and $\theta''$ are complex numbers with $|\theta|, |\theta'|,|\theta''|\le 1.$ 
    Assume that we have an exponential sum bound of the type
  \begin{equation}\label{eq:WaringExpSum}
    \max_{\xi \in \rF_{\ell}^*} \left|\sum_{n \le \tau} \el{\xi s_n} \right| \le R\,.
  \end{equation}     
     Then, combining 
     \eqr{WaringJ_k} and \eqr{WaringExpSum} we get
  \[
     T_k(\lambda)= \frac{\tau^k}{\ell} + \theta R^k = \frac{\tau^k}{\ell}\left(1+ \theta \left({R}/{\tau}\right)^k \ell\right).
  \]
     Now, if $(R/\tau)^k \ell$ goes to zero as $\ell \to \infty$, we obtain an effective asymptotic formula for $T_k(\lambda).$ In particular 
     $T_k >0$ for $\ell$ large enough. For instance, if $\tau \ge \ell^{r/2 + \varepsilon}$ we employ Korobov's bound~\eqr{Korobov}
     with $R=\ell^{r/2}$ to get
  \[
     T_k(\lambda) = \frac{\tau^k}{\ell}  \left(1+ \theta ({\ell^{r/2}}/{\tau})^{k}\ell\right) = \frac{\tau^k}{\ell}  \left(1+ \theta{\ell^{1-k\varepsilon}}\right), 
  \]     
     therefore $T_k(\lambda)=\tfrac{\tau^k}{\ell}(1 + o(1))$ for $k > 1/ \varepsilon$ in the range $\tau \ge \ell^{r/2 + \varepsilon}.$ 
     If the characteristic polynomial $\omega(x)$ of $\{s_n\}$ is irreducible with $\deg(\omega)\ge 2$ and the least period $\tau$
     satisfies $\gcd(\tau, \ell^d -1) < \tau \ell^{-\varepsilon}$ for any  divisor $d<r$ of $r,$ 
     then by  
     Corollary~\ref{coro:ThmMain} 
     we choose $R=\tau \ell^{-\delta}$
     for some positive $\delta=\delta(\varepsilon),$ to get 
  \[
     T_k(\lambda) = \frac{\tau^k}{\ell}  \left(1+ \theta ({\tau \ell^{-\delta}}/{\tau})^{k}\ell\right) =  
     \frac{\tau^k}{\ell}  \left(1+ \theta ({\ell^{1- k\delta}})\right).
  \]     
    Thus, $T_k(\lambda)>0$ when $k > 1/\delta$ and $\max_{\substack{d<r \\ d|r}}\gcd (\tau, \ell^d-1)< \tau \ell^{-\varepsilon}.$ In particular, we now need $\tau>p^{r/p(r)+\varepsilon},$ where $p(r)$ is the least prime factor of $r.$ Let us summarize the above discussion in the form of following corollary.
    \begin{corollary}\label{coro:Waring}
   Let $\ell$ be a prime number, $\varepsilon>0$ and $\{s_n\}$ be a linear recurrence sequence of order $r\ge 2$ in $\mathbb{F}_{\ell}.$
   If the characteristic polynomial $\omega(x)$ in $\rF_{\ell}[x]$ is irreducible  with $(\omega(0),\ell)=1,$ and the 
   least period $\tau$ satisfies 
 \[\max_{\substack{d<r\\ d | r}} (\tau, \ell^d -1 ) < \tau \ell^{-\varepsilon},\]  
   then there exists an integer $ k_0>0$ such that
   for any $k \ge k_0$ and every integer $\lambda$, if $T_k(\lambda)$ denotes the number of solutions of the congruence
  \[
    s_{n_1} + \cdots +s_{n_k} \equiv \lambda \hspace{-0.2cm} \pmod \ell, \quad \textrm{with} \quad 1\le n_1, \ldots, n_k \le \tau,
  \]    
    then $T_k(\lambda)= \tfrac{\tau^k}{\ell}(1 + o(1)).$
\end{corollary}  
In particular, we can now extract out the following
     \begin{corollary}\label{co:first} Let $\{s_n\}$ be a linear recurrence sequence in $\mathbb{Z},$ whose characteristic polynomial $\omega(x) \in \mathbb{Z}[x]$ is monic, separable, irreducible, and having prime degree. Then for a set of primes $\ell$ with positive density, the sequence $\{s_n\}$ is an additive basis modulo $\ell.$
      \end{corollary}
     \begin{proof} We start with writing $\mathbb{Q}_f$ to be the splitting field of $f$ and $G_f$ be $\text{Gal}\left(\mathbb{Q}_f/\mathbb{Q}\right).$ It is clear that $\deg(\omega) \mid |G_f|$ and $G_f$ is contained in the symmetric group $S_{\deg(\omega)}.$ In particular there is a $\deg(\omega)$-cycle in $G_f$ because $\deg(\omega)$ is prime. By Chebotarev's density theorem, (see \cite{Steven} for instance) the set of such primes $\ell$ for which $\omega(x)\hspace{-0.1cm}\pmod{\ell}$ is irreducible, have positive density.

 Writing 
     \[\omega(x)=\prod_{i=0}^{\deg(\omega)-1}(x-\alpha^{\ell^i})\,,\]
     we get $\omega(0)=(-\alpha)^{1+\ell+\ell^2+\cdots+\ell^{\deg(\omega)-1}}.$ We can make $(\omega(0),\ell)=1,$ 
     for all but finitely many $\ell$'s. We now need to verify the condition of Corollary~\ref{coro:ThmMain} for $d=1$ because the Galois group has prime degree. 
     We have $\gcd(\mathrm{ord}~\alpha,\ell-1)=\frac{\mathrm{ord}~\alpha}{\mathrm{ord}~\alpha^{\ell-1}}.$ Fix any $0<\varepsilon<1/2,$ and now the proof is complete if $\mathrm{ord}\left(\alpha^{\ell-1}\right)>\ell^{\varepsilon}$ holds for almost all primes $\ell.$ Note that 
     \[\alpha^{(\ell-1)t}=1 \implies  \alpha^{rt}=\left(\prod_{i=0}^{r-1} \alpha^{\ell^i}\right)^{t} \implies \alpha^{2rt}=\omega(0)^{2t}.\]
     We now consider $R(T)=\text{Res}\left( \omega(x), \prod_{t\leq T}\left(x^{2rt}-\omega(0)^{2t}\right)\right),$
     and counting the number of distinct prime factors of the resultant as in the proof of Lemma~\ref{Igor1}, we see that there exists a  $\delta$ such that 
     \begin{equation*}\label{eq:main3} 
   \max_{\xi \in \rF_{\ell}^*} \left|\sum_{n\le \tau} \el{\xi s_n }\right|\le \tau \ell^{-\delta}
 \end{equation*}
  holds, for at least   $\frac{1}{\deg(\omega)}\pi(y)+O(y^{2\varepsilon})$ many primes $\ell \leq y.$ Once we have the estimate as above, the proof follows immediately following the discussion in the previous page.
\end{proof}
\subsection{Waring-type problems for modular forms}
   Let us recall our discussion from the introduction on Waring problem for modular forms. In this section, we are assuming the modular form is a newform without $CM$. Fix any $0<\varepsilon<\frac{1}{2},$ say $\varepsilon=\frac{1}{3}.$ Then taking $\delta:=\delta(\varepsilon)$ as in Theorem~\ref{thm:main1-1}, the following estimate
\begin{equation*}
    \max_{\xi \in \rF_{\ell}^*} \left|\sum_{n\le \tau}\el{\xi a(p^n) }\right| \leq 
    \tau \ell^{-\delta},
  \end{equation*}
holds for almost all primes $p$ and $\ell.$ The discussion in Section~\ref{subse:waring} shows that $T_s(\lambda)>0$ for any $\lambda \in \mathbb{F}_{\ell},$ and $s>1/\delta,$ where $T_s(\lambda)$ is the number of solutions of the congruence
  \[ 
    a(p^{n_1}) + \cdots +a(p^{n_s}) \equiv \lambda \hspace{-0.2cm} \pmod \ell,  \quad \textrm{with} \quad 1\le n_1, \ldots, n_s \le \tau.
  \]
Moreover, this $s$ does not depend on the choice of the eigenform because $\delta$ does not. More precisely, we have the following result.
\begin{corollary}\label{le:later} Let $f$ be a newform without CM and with rational Fourier coefficients. We say, a proposition $\mathcal{Q}_{f}(p,\ell,s)$ is true if and only if, any element of $\mathbb{F}_{\ell}$ can be written as a sum of at most $s$ elements of the set $\{a(p^n)\}_{n\geq 0}.$ Then there is an absolute constant $s_0$ such that $Q_{f}(p,\ell,s_0)$ is true for almost all primes $p$ and $\ell$. Moreover, $s_0$ does not depend on the choice of $f.$

\end{corollary}
Moreover, it follows from Theorem~\ref{thm:main2-1} that

\begin{corollary}\label{le:later2} Suppose the newform is without CM and with integer Fourier coefficients. Then there exists an absolute constant $s_0$ such that, for any large prime $\ell$ satisfying the coprimality condition $(\ell-1,k-1)=1,$ the proposition $Q_{f}(p,\ell,s_0)$ is true for a set of primes $p$ with density at least $1+O\left(\frac{1}{\sqrt{\ell}}\right).$ Moreover, $s_0$ does not depend on the choice of $f.$

\end{corollary}

\subsection{Bound of non-linearity of a linear recurrence sequence}   Let $\{s_n\}$ be  a linear recurrence sequence modulo $\ell$ as in~\eqref{eq:RecurrenceSeq}
   with order $r,$ $(a_0,\ell)=1$ and period $\tau$. For  $0 \le b \le \ell^r-1 $, let us define the sum
 \[
   W(b)= \sum_{n \le \tau} \el{s_n + { \inner{b,n}}},
 \]  
   where $\inner{b,n}$ denotes the inner product $\inner{b,n}=b_0 n_0 + \cdots + b_{r-1}n_{r-1}$
   assuming that $0\le b,n\le \ell^r-1$ are written in its $\ell$--ary expansion 
 \[
   b=b_0 + b_1 \ell + \cdots + b_{r-1}\ell^{r-1}, \qquad n=n_0 + n_1 \ell + \cdots + n_{r-1}\ell^{r-1}.
 \]  
   Bounds for $W(b)$ have cryptographic significance, see~\cite{Shparlinski-Winterhof} and 
   references therein. Shparlinski and Winterhof~\cite[Theorem 1]{Shparlinski-Winterhof} 
   proved that
 \[
   \max_{0 \le b \le \ell^r-1} |W(b)| \ll \tau^{3/4} r^{1/4} \ell^{r/8},
 \]  
   whenever the characteristic polynomial of $\{s_n\}$ is irreducible. Such bound is asymptotically 
   effective if $r\ell^{r/2}/\tau \to 0.$ Combining Corollary~\ref{coro:ThmMain} and the ideas of
   Shparlinski and Winterhof we are able to improve such bound for a large class of linear recurrence sequences 
   in the range $\tau > \ell^{\varepsilon}.$ 
   For example, assuming hypothesis of Corollary~\ref{coro:ThmMain}, if $r$ is fixed then  
   $|W(b)|\ll \tau \ell^{-\delta'}$ as $\ell \to \infty$
   for some  $\delta'>0.$ In general we get $|W(b)|=o(\tau)$ if 
   $r \log \ell / \ell^{\delta'} \to 0$ as $\ell \to \infty.$ More precisely
   \begin{corollary}\label{coro:Non-linearity}
   Let $\ell$ be a prime number, $\varepsilon>0$ and $\{s_n\}$ be a linear recurrence sequence of order $r\ge1.$
   If the characteristic polynomial $f(x)$ in $\rF_{\ell}[x]$ is irreducible polynomial with $(f(0),\ell)=1,$ and the least period $\tau$ satisfies 
 \[ \tau > \ell^{\varepsilon}, \quad \textrm{and} \quad 
    \max_{\substack{d<r \\ d | r}} (\tau, \ell^d -1 ) < \tau \ell^{-\varepsilon}, \]  
   then there exists a $\delta=\delta(\varepsilon)>0$ such that
 \[\max_{0 \le b \le \ell^r-1} \left| \sum_{n \le \tau} \el{s_n + { \inner{b,n}}} \right | \le 
   \tau \ell^{-\delta/4} (r \log \ell)^{1/4}\left(1 + \ell^{-\delta/4} (r \log \ell)^{1/4} \right).\]  
 \end{corollary}   
 \begin{proof}  
   The proof follows the same steps given by Shparlinski and Winterhof~\cite[Theorem 1]{Shparlinski-Winterhof}.
   We just need to employ the bound given by Corollary~\ref{coro:ThmMain} instead of Korobov's bound.
\end{proof}

\noindent{\textbf{Note.}}
we have an improvement on the bound due to Shparlinski and Winterhof if 
    \[\tau \leq \frac{\ell^{r/2+\delta}}{\log \ell}.\]
Clearly there are many such cases, for instance, one can consider any element in $\mathbb{F}_{\ell}^{*}$ of order smaller than $\ell^{\frac{1}{2}}.$

\section*{Acknowledgements}
 Authors would like to thank Igor Shparlinski, through personal communications, and Will Sawin, through MathOverflow, for useful suggestions during the writing of the article. 
 
 In addition, authors would like to thank University of G\"ottingen, where most of the discussion and work took place, 
for its hospitality and extend their thanks to Harald Helfgott for the support and encouragement. 

The work of the first author is supported by the ERC Consolidator grants  648329 and 681207 and the second author is
 supported by the ERC Consolidator grant 648329.

\nocite{}
\bibliographystyle{abbrv}
\bibliography{BBG}

\end{document}